\documentclass[11pt,reqno]{amsart}
\usepackage{amssymb,amsmath,amsfonts}
\usepackage{mathrsfs}
\usepackage[margin=1in]{geometry}
\usepackage{graphicx}
\usepackage{relsize}
\usepackage{color,xcolor}
\definecolor{cobalt}{RGB}{61,99,181}

\usepackage[colorlinks, citecolor=cobalt, linkcolor=cobalt]{hyperref}

\usepackage{chngcntr}
\usepackage{extarrows}

\allowdisplaybreaks

\newtheorem{thm}{Theorem}[section]

\newtheorem{cor}[thm]{Corollary}
\newtheorem{lem}[thm]{Lemma}

\newtheorem{exam}[thm]{Example}

\newtheorem{prop}[thm]{Proposition}

\numberwithin{equation}{section}

\def\D{\mathbb{D}}

\newcommand{\N}{\mathbb{N}}

\newcommand{\R}{\mathbb{R}}
\newcommand{\C}{\mathbb{C}}

\makeatletter

\newcommand{\Rmnum}[1]{\expandafter\@slowromancap\romannumeral #1@}
\@namedef{subjclassname@2020}{\textup{2020} Mathematics Subject Classification}
\makeatother

\begin{document}

\title[Weighted  composition operators]{Weighted composition operators on the Fock space}
	
	\author[Xianchi Tian]{Xianchi Tian}
	\address{\textsuperscript{1} College of Mathematics and Statistics, Chongqing University, Chongqing 401331, P. R. China}
	\email{202306021073t@stu.cqu.edu.cn}
	
	\date{\today}
	\thanks{}
	
\begin{abstract}
		In this paper, we  discuss the compactness, invertibility and  isometry of the operators on the Fock space formed by multiplying a weighted composition operator  with the adjoint of another weighted composition operator to form either $W_{g, \psi}W_{f, \varphi}^*$ or $W_{g, \psi}^*W_{f, \varphi}$.
	\end{abstract}

\maketitle

\section{Introduction}

 Let $dA(z)$ denote the  Lebesgue area measure on the complex plane $\mathbb{C}$ and  $d\mu(z)=\frac{1}{\pi}e^{-|z|^2}dA(z)$  denote the Gaussian measure on $\mathbb C$. The Fock space $\mathcal F^2$ consists of all entire functions $f$ in $L^2(\mathbb C, d\mu)$. One can show easily that $\mathcal F^2$ is a closed subspace of $L^2(\mathbb C, d\mu)$. Thus, the Fock space $\mathcal F^2$ is a Hilbert space with the inner product:
$$\langle f, g\rangle=\int_{\mathbb C}f(z)\overline{g(z)}d\mu(z).$$
In addition, $\mathcal F^2$  is also  a reproducing kernel Hilbert space and the reproducing kernel at $z\in \mathbb C$ is given by
$$K_z(w)=e^{\overline{z}w},  \ \ \ \  w\in \mathbb C.$$
Clearly, the normalized reproducing kernel at $z$ is given by
$$k_z(w)=\frac{K_z(w)}{\|K_z\|}=e^{\overline{z}w-\frac{|z|^2}{2}},  \ \ \ \ w\in \mathbb C.$$
For more details about the Fock space $\mathcal{F}^2$ and operator theory   on  this space, we refer to the classical book  \cite{ZK}.

Let $f$ and $\varphi$ be  entire functions.   The \emph{maximal weighted composition operator} $W_{f, \varphi}$ on the Fock space $\mathcal F^2$ is defined by $$W_{f, \varphi}h(z)=f(z)h(\varphi(z))$$ for $h\in \mathcal{F}^2$ and $z\in \C$ such that $W_{f, \varphi}h$ belongs to $\mathcal{F}^2$. In particular, when $f=1$, $W_{f, \varphi}$ reduces to the (unweighted) composition operator $C_{\varphi}$. Concerning  the boundedness, compactness, supercyclicity, unitarity, invertibility, Schatten  class and closed-range property of a single weighted composition operator on the Fock space $\mathcal{F}^2$, one can consult \cite {TC},  \cite{HW},  \cite{LT},  \cite{MT1}, \cite{MT2},   \cite{MT3}, \cite{UK} and \cite{Z}.

For the densely defined operator $W_{f,\varphi}$, the domain
of its adjoint can be described as follows.
A function $h\in\mathcal F^2$ belongs to $D(W_{f,\varphi}^*)$
if and only if there exists a constant $C_h>0$ such that
$$
|\langle W_{f,\varphi}u,h\rangle|
\leqslant C_h\|u\|,
\qquad u\in D(W_{f,\varphi}).
$$
For such $h$, the Riesz representation theorem gives a unique
$v\in\mathcal F^2$ satisfying
$$
\langle W_{f,\varphi}u,h\rangle=\langle u,v\rangle,
\qquad u\in D(W_{f,\varphi}),
$$
and we define $W_{f,\varphi}^*h=v$.
Note that the adjoint $W_{f,\varphi}^*$ need not be bounded.

Hai \cite{PVH} studied unbounded weighted composition operators on the Fock space,
characterizing their $\mathcal C$-selfadjointness with respect to weighted
composition conjugations, as well as their Hermitian, normal and cohyponormal
properties. For related results on unbounded weighted composition operators and
generalized weighted composition operators on Fock spaces and related
spaces of entire functions, see \cite{MG-generalized-Fock,HT-multivariable,HR-MittagLeffler}.

The study of the product of a composition operator and the adjoint of another composition operator originated in the literature \cite{JC, CD1}.
The authors established  necessary and sufficient conditions, under various hypotheses on the inducing maps, for the operators $C_{\psi}^*C_{\varphi}$ and $C_{\varphi}C_{\psi}^*$  to be compact on the  Hardy space, where $\varphi$ and $\psi$ are analytic self maps of the open  unit disk $\D$.  They later extended their results  to the case of  weighted Bergman spaces and characterized the compactness of  the products $C_{\psi}^*C_{\varphi}$ and $C_{\varphi}C_{\psi}^*$; see \cite{CD} for more details.

In 2014, Clifford, Le and Wiggins \cite{CL} completely characterized when the operator $C_\varphi C_\psi^*$  on the Hardy space is invertible, isometric and unitary, and when the operator
$C_{\psi}^* C_\varphi$  is isometric and unitary. In 2018 and 2019, Shaabani \cite{SH,SH2} investigated the invertibility, Fredholmness and closed-range property for the products $W_{g, \psi}W_{f, \varphi}^*$ and  $W_{g, \psi}^*W_{f, \varphi}$ on the Hardy space and weighted Bergman spaces.
 Clifford et al. \cite{CM} considered  the question of whether the operator $C_{\varphi}^*C_{\psi}$ can be compact or finite-rank on the Hardy space in the case that $\varphi$ and $\psi$ are inner functions. Recently, the Schatten class membership of the product of a composition operator and the adjoint of another composition operator on the Hardy and Bergman spaces is investigated by Leng and Zhao \cite{LZ}. Although discussing the compactness, unitarity and invertibility of the operators $W_{g,\psi}W_{f, \varphi}^*$ and $W_{g,\psi}^*W_{f, \varphi}$ on the Fock space is quite natural, there has been very little known about  this aspect.

Inspired by the aforementioned work, we investigate the compactness,
invertibility, isometry and unitarity of products formed by a weighted
composition operator and the adjoint of another on the Fock space.
Since Littlewood's subordination principle does not apply to entire
functions, composition operators on this space are not automatically
bounded.

In Section \ref{S3}, we assume that both $W_{f,\varphi}$ and
$W_{g,\psi}$ are bounded and study the products
$W_{g,\psi}W_{f,\varphi}^*$ and $W_{g,\psi}^*W_{f,\varphi}$.
By Theorem \ref{U}, each of these products is unitary if and only if
$|\varphi^{\prime}(z)\psi^{\prime}(z)|=1$ for all $z\in\C$
and
$|f(0)g(0)|e^{\frac{|\varphi(0)|^2+|\psi(0)|^2}{2}}=1.$ In addition, we show in
Theorem \ref{I} that the  invertibility of $W_{g,\psi}W_{f,\varphi}^*$ and $W_{g,\psi}^*W_{f,\varphi}$ is equivalent, and that  is also equivalent to the condition
$|\varphi^{\prime}(z)\psi^{\prime}(z)|=1$ for all $z\in\C$.
Moreover, we conclude in Theorem \ref{M5} that $W_{g,\psi}W_{f,\varphi}^*$ is unitary if and only if it is isometric, and likewise for the operator  $W_{g,\psi}^*W_{f,\varphi}$.
We also obtain necessary conditions and sufficient conditions for
compactness in Propositions \ref{T2} and \ref{T3}.
When $f=g=1$, these results give complete characterizations of
unitarity and compactness for the products
$C_{\psi}C_{\varphi}^*$ and $C_{\psi}^*C_{\varphi}$;
see Corollaries \ref{C1} and \ref{C2}.

In Section \ref{S4}, we allow one factor to be unbounded and study the unique bounded extension of the densely defined natural product $W_{g,\psi}W_{f,\varphi}^{*}$. Under this assumption, Theorem \ref{ext-affine} ensures that both nonconstant inducing maps are affine. In addition, Theorems  \ref{ext-unitary} and \ref{ext-invertible} characterize isometry, unitarity and invertibility, while Propositions \ref{ext-Gaussian-compact} and \ref{ext-finite-zeros} establish compactness criteria. However, we have not obtained analogous results for bounded extensions of the reverse natural product $W_{g,\psi}^{*}W_{f,\varphi}$ under this weaker boundedness assumption. This case is left for further study.

\section{Preliminary}\label{S2}

In this section, we shall include  some definitions and known results which will be frequently used  in the next two sections. The following theorem gives  complete characterizations of the boundedness  and compactness of a weighted composition operator on the Fock space $\mathcal F^2$.

\begin{thm}\label{L1}\emph{(\cite[Theorems 2.2 and 2.4]{LT})}
  Let $f$ and $\varphi$ be entire functions on $\C$ such that $f$ is not identically zero. Then we have:\\
  $(1)$ $W_{f, \varphi}$ is bounded on $\mathcal{F}^2$ if and only if $f$ belongs to $\mathcal{F}^2$, $\varphi(z)=\lambda z+\varphi(0)$ with $|\lambda|\leqslant 1$ and
  $$ M(f, \varphi):=\sup\big\{|f(z)|^2e^{(|\varphi(z)|^2-|z|^2)}: z\in \C\big\}< \infty;$$
 $(2)$ $W_{f, \varphi}$ is compact on $\mathcal{F}^2$ if and only if $\varphi(z)=\varphi(0)+\lambda z$ for some $|\lambda|<1$ and
$$\lim_{|z|\rightarrow \infty} |f(z)|^2e^{|\varphi(z)|^2-|z|^2}=0.$$
 \end{thm}
 
 We also recall the following lemma, which gives the explicit form of the weight when $|\varphi'(z)|=1$.
 \begin{lem}\label{|a|=1}\emph{(\cite[Proposition 2.1]{LT})}
Let $f$ and $\varphi$ be two entire functions on $\mathbb{C}$
such that $f$ is not identically zero. Suppose that there is
a positive constant $M$ such that
$$|f(z)|^2 e^{|\varphi(z)|^2-|z|^2}
\leqslant M\quad \text{for all } z\in\mathbb{C}.$$
Then $\varphi(z)=\varphi(0)+\lambda z$ for some $|\lambda|\leqslant 1$.
If $|\lambda|=1$, then $f(z)=f(0)e^{-\overline{\beta}z}$,
where $\beta=\overline{\lambda}\varphi(0)$.
\end{lem}

Let $f$ be an entire function. For any $r>0$, let
 $$M(r)=M_{f}(r)=\sup\{|f(z)|: |z|=r\}.$$
Then the order $\rho$ of $f$ is defined as
 $$\rho=\rho(f):=\limsup_{r\rightarrow \infty}\frac{\log \log M(r)}{\log r}.$$
 It is clear that $0\leqslant \rho \leqslant \infty$.
 When $0 <\rho <\infty$,  the type $\sigma$ of $f$ is defined by
 $$\sigma=\limsup_{r \rightarrow \infty}\frac{\log M(r)}{r^{\rho}}.$$

The next theorem gives a sufficient condition for the weighted composition operator $W_{f,\varphi}$ to be compact on the Fock space, which is useful for us to study the compactness of the products  $W_{g, \psi}W_{f, \varphi}^*$ and  $W_{g, \psi}^*W_{f, \varphi}$.

\begin{thm}\label{L3}\emph{(\cite[Theorem 3.1]{TC})}
  Let  $\varphi(z)=a+\lambda z$ with $|\lambda|<1$. If $f$ has order strictly less than $2$, or if $f$ has order $2$ and type strictly less than $\frac{\beta}{2}$, where $\beta=1-|\lambda|^2$, then the weighted composition operator $W_{f, \varphi}$ is  compact  on $\mathcal{F}^2$.
\end{thm}

In order to discuss when the product of a weighted composition operator and the adjoint of another weighted composition operator is unitary and invertible on the Fock space, we need the following  characterizations for the unitarity and invertibility  of a single weighted composition  operator.

\begin{thm}\label{L6}\emph{(\cite[Corollary 1.2]{Z})}
  Let $f$ and $\varphi$ be entire functions on $\C$. Then $W_{f,\varphi}$ is a unitary operator on $\mathcal{F}^2$ if and only if there exist constants $a, b, \alpha$ with $|a|=|\alpha|=1$ such that $\varphi(z)= az+b$ and $f(z)= \alpha k_{-\overline{a}b}(z)$.
\end{thm}

\begin{thm}\label{L7}\emph{(\cite[Theorem 5.1]{HW})}
   Let $f$ and $\varphi$ be entire functions such that $f$ is not identically zero. Suppose that $W_{f, \varphi}$ is  bounded on $\mathcal{F}^2$, and hence $\varphi(z)=az+b$ with $|a|\leqslant 1$. Then the following statements are equivalent:\\
  \emph{ (i)} $W_{f, \varphi}$ is surjective on $\mathcal{F}^2$;\\
  \emph{ (ii)} $W_{f, \varphi}$ is invertible on $\mathcal{F}^2$;\\
   \emph{(iii)} $\varphi(z)=az+b$ with $|a|= 1$.
\end{thm}

The following theorem gives a necessary and sufficient condition for a weighted composition operator to have closed range.

\begin{thm}\label{L8}\emph{(\cite[Theorem 1.2]{MT3})}
   Let $f$ and $\varphi$ be entire functions such that $f$ is not identically zero. Assume that $W_{f, \varphi}$ is  bounded on $\mathcal{F}^2$, and hence $\varphi(z)=az+b$ with $|a|\leqslant 1$.  Then $W_{f, \varphi}$ has a closed range if and only if either $a=0$, or $|a|=1$  and $f(z)=f(0)K_{-\overline{a}b}(z)$ with $f(0)\neq 0.$
\end{thm}

Let $W_{f,\varphi}$ and $W_{g,\psi}$ be densely defined maximal weighted composition operators, where $f,g,\varphi,\psi$ are entire functions. In the following, for a maximal weighted composition operator $W_{f, \varphi}$,
 $$D(W_{f, \varphi})=\{h\in\mathcal F^2: f\cdot (h\circ \varphi)\in \mathcal F^2\}$$
 denotes the maximal domain of $W_{f, \varphi}$.
In the rest of this section, we assume that one of $W_{f,\varphi}$ and $W_{g,\psi}$ is bounded and the other one is unbounded. The natural product $T_0=W_{g,\psi}W_{f,\varphi}^*$ is assumed to be densely defined and to satisfy $\|T_0h\|\leqslant M\|h\|$ for $h\in D(T_0)$ and some constant $M>0$, where $$D(T_0)=\big\{h\in D(W_{f,\varphi}^*):W_{f,\varphi}^*h\in D(W_{g,\psi})\big\}$$ denotes the natural domain of $T_0$.
 We write $T=\overline{T_0}$ for its unique bounded extension to $\mathcal F^2$. In Section \ref{S4},  we will study the compactness, invertibility, isometry and unitarity of  this extension.

 Next, let us present the relevant results concerning unbounded weighted composition operators on the Fock space $\mathcal F^2$.

\begin{lem}\label{ext-adjoint-kernel}\emph{(\cite[Lemma 3.1(1)]{PVH})}
Let $W_{\psi,\varphi}$ be a densely defined weighted composition operator induced by two entire functions $\psi,\varphi$. Then, for every $z\in\C$, we have $K_z\in\operatorname{dom}(W_{\psi,\varphi}^*)$, and
$$W_{\psi,\varphi}^*K_z=\overline{\psi(z)}K_{\varphi(z)}.$$
\end{lem}

\begin{lem}\label{ext-maximal-closed}\emph{(\cite[Proposition 3.4]{PVH})}
Every maximal weighted composition operator is closed on the Fock space $\mathcal F^2$.
\end{lem}

\begin{lem}\label{ext-exponential-fock}\emph{(\cite[Lemma 2.1]{PVH})}
  If the function $h$ is entire, then $e^h\in \mathcal F^2$ if and only if $h(z)=\alpha z^2 +\beta z +\gamma$ with $|\alpha|<\frac12$.
\end{lem}

We determine the domain of the natural product $T_0$ in the following lemma.

\begin{lem}\label{ext-domain}
Let $f,g,\varphi,\psi$ be entire functions and
let $W_{f,\varphi}$ and $W_{g,\psi}$ be densely defined
maximal weighted composition operators on $\mathcal F^2$.
Suppose that exactly one of $W_{f,\varphi}$ and $W_{g,\psi}$ is bounded and that
the natural product $T_0=W_{g,\psi}W_{f,\varphi}^*$ is densely defined and admits
a unique bounded extension $T$ to $\mathcal F^2$. Then
$$D(T_0)=D(W_{f,\varphi}^*).$$
In particular, $K_w\in D(T_0)$ for every $w\in\C$, and
\begin{align}\label{ext-kernel}
TK_w(z)=\overline{f(w)}g(z)
e^{\overline{\varphi(w)}\psi(z)},\qquad z,w\in\C.
\end{align}
\end{lem}

\begin{proof}
If $W_{g,\psi}$ is bounded, then $D(W_{g,\psi})=\mathcal F^2$. Hence
\begin{align*}
D(T_0)&=\{h\in D(W_{f,\varphi}^*):W_{f,\varphi}^*h\in D(W_{g,\psi})\}=D(W_{f,\varphi}^*).
\end{align*}

Suppose that $W_{f,\varphi}$ is bounded. Then $W_{f,\varphi}^*$ is bounded and
$D(W_{f,\varphi}^*)=\mathcal F^2$. Since $\overline{D(T_0)}=\mathcal F^2$, for any $h\in\mathcal F^2$,
there exists  a sequence
$\{h_n\}\subseteq D(T_0)$ such that $h_n$ converges to $h$ in
$\mathcal F^2$. Since $W_{f,\varphi}^*$ and $T$ are bounded, we have
\begin{align*}
\lim_{n\rightarrow \infty}W_{f,\varphi}^*h_n=W_{f,\varphi}^*h
\end{align*}
and
\begin{align*}
\lim_{n\rightarrow \infty}W_{g,\psi}W_{f,\varphi}^*h_n
=\lim_{n\rightarrow \infty}T_0h_n=\lim_{n\rightarrow \infty}Th_n=Th.
\end{align*}
 By Lemma \ref{ext-maximal-closed}, $W_{g,\psi}$ is closed.  Since $W_{f,\varphi}^*h_n\in D(W_{g,\psi})$, we have $W_{f,\varphi}^*h\in D(W_{g,\psi})$. Moreover,
$W_{g,\psi}W_{f,\varphi}^*h=Th$. Thus $h\in D(T_0)$, and consequently
$$D(T_0)=\mathcal F^2=D(W_{f,\varphi}^*).$$

Finally, we prove \eqref{ext-kernel}. By Lemma \ref{ext-adjoint-kernel}, for every $w\in\C$,
$$K_w\in D(W_{f,\varphi}^*),\qquad
W_{f,\varphi}^*K_w=\overline{f(w)}K_{\varphi(w)}.$$
The equality $D(T_0)=D(W_{f,\varphi}^*)$ gives $K_w\in D(T_0)$.
Therefore, for all $z,w\in\C$,
\begin{align*}
TK_w(z)&=T_0K_w(z)=(W_{g,\psi}W_{f,\varphi}^*K_w)(z)\\
&=g(z)(W_{f,\varphi}^*K_w)(\psi(z))\\
&=\overline{f(w)}g(z)e^{\overline{\varphi(w)}\psi(z)},
\end{align*}
which proves \eqref{ext-kernel}.
\end{proof}

By \eqref{ext-kernel}, if either $\varphi$ or $\psi$ is constant, then $T$ has rank at most one. Thus $T$ is compact and is neither invertible nor isometric in these cases. In the following, we assume that both $\varphi$ and $\psi$ are nonconstant.

We recall two results on entire functions of finite order.
\begin{lem}\label{zero-summability}\emph{(\cite[Chapter 5, Theorem 2.1]{SS})}
Let $f$ be an entire function of order at most   $\rho$.  If $z_1,z_2,\cdots$ denote the zeros of $f$, with each $z_k\neq 0$, then for every $s>\rho$ we have
$$\sum_{k=1}^\infty\frac{1}{|z_k|^s}<\infty.$$
\end{lem}

\vspace{0.4mm}
For each integer $k\geqslant 0$, we define \emph{canonical factors} by
$$E_0(z)=1-z\ \ \ \text{and}\ \ \ E_{k}(z)=(1-z)\exp\Big\{z+\frac{z^2}{2}+\cdots+\frac{z^k}{k}\Big\},\ \ \ k\geqslant 1.$$
The integer $k$ is called the degree of the canonical factor. We shall use the following form of Hadamard's factorization theorem.
\begin{thm}\label{hadamard-factorization}\emph{(\cite[Chapter 5, Theorem 5.1]{SS})}
Suppose  $f$ is  entire and has growth order $\rho_0$. Let $k$ be the integer so that $k\leqslant\rho_0<k+1$. If $a_1, a_2, \cdots$ denote the nonzero zeros of $f$,
then
  $$f(z)=e^{P(z)}z^m\prod_{n=1}^\infty E_k\left(\frac{z}{a_n}\right),$$
  where $P$ is a polynomial of degree  not exceeding $k$, and $m$ is the order of the zero of $f$ at $z=0$.
\end{thm}
The product in  Theorem \ref{hadamard-factorization} is finite when there are finitely many nonzero zeros, and an empty product is understood to be $1$.

We conclude this section with two classical results on entire functions.

\begin{lem}\label{lem:holomorphic}\emph{(\cite[Chapter 5, Proposition 3.2]{SS})}
  Suppose $\{F_n\}_{n=1}^\infty$ is a sequence of holomorphic functions on the open set $\Omega$. If there exist constants $c_n>0$ such that
  $$\sum_{n=1}^\infty c_n<\infty\ \ \ \ \text{and}\ \ \ \ |F_n(z)-1|\leqslant c_n\ \ \ \ (z\in\Omega,\ n\geqslant 1),$$
  then the product $\prod\limits_{n=1}^\infty F_n(z)$ converges uniformly in $\Omega$ to a holomorphic function $F(z)$.
\end{lem}

\begin{lem}\label{lem:canonical-factor-estimate}\emph{(\cite[Chapter 5, Lemma 4.2]{SS})}
  If $|z|\leqslant \frac12,$ then $|1-E_k(z)|\leqslant c|z|^{k+1}$ for some $c>0.$
\end{lem}

\section{Products with two bounded factors} \label{S3}
Throughout this section, we assume that both  $W_{g, \psi}$ and $W_{f, \varphi}$ are bounded on the Fock space $\mathcal{F}^2$. Under this assumption, we  discuss the compactness, invertibility and  isometry of  $W_{g, \psi}W_{f, \varphi}^*$ and $W_{g, \psi}^*W_{f, \varphi}$.

Let us begin with the following characterization of when the operators $W_{g, \psi}W_{f, \varphi}^*$ and $W_{g, \psi}^*W_{f, \varphi}$ are unitary on  $\mathcal{F}^2$.
\begin{thm}\label{U}
   Suppose that $f, g,\varphi$ and $\psi$ are  entire functions such that $f(0)g(0)\neq 0$. Assume that $W_{g, \psi}$ and $W_{f, \varphi}$ are bounded on the Fock space $\mathcal{F}^2$. Then the following are equivalent:\\
  $(1)$ $W_{g, \psi}W_{f, \varphi}^*$ is unitary on $\mathcal{F}^2$;\\
   $(2)$ $W_{g, \psi}^*W_{f, \varphi}$ is unitary on $\mathcal{F}^2$;\\
   $(3)$ $|\varphi^{\prime}(z)\psi^{\prime}(z)|=1$ for all $z\in \mathbb C$ and  $|f(0)g(0)|e^{\frac{|\varphi(0)|^2+|\psi(0)|^2}{2}}=1$.
\end{thm}

   \begin{proof}
    Since $W_{g, \psi}$ and $W_{f, \varphi}$ are both bounded on the Fock space $\mathcal F^2$,  we have by (1) of Theorem  \ref{L1} that $\varphi(z)=az+b$ with $|a|\leqslant 1$ , $\psi(z)=cz+d$ with $|c|\leqslant 1$, and $f, g$ belong to $\mathcal{F}^2$ with $ M(f, \varphi)<\infty$ and $M(g, \psi)< \infty$.

    We first show that $|a|=|c|=1$ provided that (1) or (2) holds.  If  $W_{g, \psi}W_{f, \varphi}^*$ is unitary, then we have that $W_{g, \psi}$ and $W_{f, \varphi}$ are both surjective. It follows from Theorem \ref{L7} that $|a|=|c|=1$.

     Assuming that $W_{g, \psi}^*W_{f, \varphi}$ is unitary,  we need  to show  $|a|=|c|=1$.
Denoting $A=W_{g, \psi}W_{g, \psi}^*$, we have
  $$I=(W_{g, \psi}^*W_{f, \varphi})^*W_{g, \psi}^*W_{f, \varphi}=W_{f, \varphi}^*AW_{f, \varphi},$$
  where $I$ is the identity operator on $\mathcal F^2$.
 Since $A$ is positive, we have
  \begin{align*}
    \begin{split}
      \|A\|\|W_{f, \varphi}h\|^2&\geqslant \langle AW_{f, \varphi}h, W_{f, \varphi}h\rangle \\
      &=\langle W_{f, \varphi}^*AW_{f, \varphi}h, h\rangle \\
      &=\langle h, h\rangle=\|h\|^2,\ \ \ \ h\in \mathcal F^2,
    \end{split}
  \end{align*}
  which yields  $W_{f,\varphi}$ is bounded below on $\mathcal F^2$, and hence it has closed range.
It follows  from Theorem \ref{L8} that either $a=0$ or $|a|=1$. If $a=0$, then we have
\begin{align*}
  \begin{split}
  W_{f, \varphi}h(w)&=f(w)h(\varphi(w))=f(w)h(b)\\
  &=\langle h, K_b\rangle f(w)=(f\otimes K_b)h(w)
  \end{split}
\end{align*}
for any $h\in \mathcal{F}^2$ and $w\in \C$.  This implies that $W_{f, \varphi}=f\otimes K_b$ is a rank-one operator, and so is compact. This contradicts the condition that $W_{f, \varphi}$ is bounded below. Therefore,  we conclude that $|a|=1.$ Similarly, using
$$W_{g, \psi}^*W_{f, \varphi}W_{f, \varphi}^*W_{g, \psi}=I$$
we can also obtain  that $|c|=1$.

     From the above discussion,  we conclude that  $|a|=|c|=1$ if (1) or (2) holds.  It follows from Lemma \ref{|a|=1}  that $$f(w)=f(0)e^{-a\overline{b}w}\ \ \ \ \mathrm{and}\ \ \ \ g(w)=g(0)e^{-c\overline{d}w}.$$ Consequently,
          \begin{align*}
          \begin{split}
           f(z)&=f(0)K_{-\overline{a}b}(z)=f(0)\|K_{-\overline{a}b}\|k_{-\overline{a}b}(z)\\
           &=f(0)e^{\frac{|a\overline{b}|^2}{2}}k_{-\overline{a}b}(z)=f(0)e^{\frac{|b|^2}{2}}k_{-\overline{a}b}(z).
          \end{split}
        \end{align*}
       Similarly $$g(z)=g(0)e^{\frac{|d|^2}{2}}k_{-\overline{c}d}(z).$$
  This yields that  $$W_{f, \varphi}=f(0)e^{\frac{|b|^2}{2}}W_{k_{-\overline{a}b}, \varphi} \ \ \ \ \text{and}\ \ \ \  W_{g, \psi}=g(0)e^{\frac{|d|^2}{2}}W_{k_{-\overline{c}d}, \psi}.$$  Under the assumption that  $W_{g, \psi}W_{f, \varphi}^*$ or $W_{g, \psi}^*W_{f, \varphi}$ is unitary, we have by Theorem \ref{L6} that $$|f(0)g(0)|e^{\frac{|b|^2+|d|^2}{2}}=1,$$ which gives $(3)$.

  Next, we will show that $(3)\Rightarrow (1)$. Suppose that $|ac|=1$ and $|f(0)g(0)|e^{\frac{|b|^2+|d|^2}{2}}=1$. From the proof of $(1)\Rightarrow (3)$ we have that $$W_{f, \varphi}=f(0)e^{\frac{|b|^2}{2}}W_{k_{-\overline{a}b},\varphi} \ \ \ \ \text{and}\ \ \ \  W_{g, \psi}=g(0)e^{\frac{|d|^2}{2}}W_{k_{-\overline{c}d}, \psi},$$ which implies that
   $$W_{g, \psi}W_{f, \varphi}^*=\overline{f(0)}g(0)e^{\frac{|b|^2+|d|^2}{2}}W_{k_{-\overline{c}d}, \psi}W_{k_{-\overline{a}b}, \varphi}^*.$$
   Since Theorem \ref{L6} tells us that $W_{k_{-\overline{a}b}, \varphi}$ and $W_{k_{-\overline{c}d}, \psi}$ are both unitary,  we obtain that $W_{g, \psi}W_{f, \varphi}^*$ is also unitary.

  Finally, we can prove  $(3)\Rightarrow (2)$ by using the same method as the one used in the proof of $(3)\Rightarrow (1)$. This completes the proof of the theorem.
\end{proof}

For unweighted composition operators, we show in the following corollary that the  product   $C_{\psi}C_{\varphi}^*$ (or $C_{\psi}^*C_{\varphi}$) is unitary if and only if  $C_{\varphi}$ and $C_{\psi}$ are.

\begin{cor}\label{C1}
   Suppose   $C_{\varphi}$ and $C_{\psi}$ are bounded composition operators on $\mathcal{F}^2$. Then the following are equivalent:\\
   $(1)$ $C_{\psi}C_{\varphi}^*$ is unitary on $\mathcal{F}^2$;\\
   $(2)$ $C_{\psi}^*C_{\varphi}$ is unitary on $\mathcal{F}^2$;\\
   $(3)$ $C_{\varphi}$ and $C_{\psi}$ are both unitary;\\
   $(4)$  $\varphi(z)=az$ and $\psi(z)=cz$, where $|a|=|c|=1$.
\end{cor}

\begin{proof}

  Note that $(1)\Leftrightarrow (2)$ follows from Theorem \ref{U}. Suppose $(1)$ or $(2)$ holds. It follows from Theorem $\ref{U}$ and its proof that $|a|=|c|=1$ and $e^{\frac{|b|^2+|d|^2}{2}}=1$, where $a=\varphi^{\prime}(z)$, $c=\psi^{\prime}(z)$, $b=\varphi(0)$ and $d=\psi(0)$. This yields that $b=d=0$,  and so $\varphi(z)=az$ and $\psi(z)=cz$.  By Theorem \ref{L6},  we deduce that $C_{\varphi}$ and $C_{\psi}$ are both unitary. Thus we establish that $(1)\Leftrightarrow (3)$ and $(2)\Leftrightarrow(3)$. Finally, using Theorem \ref{L6} again we obtain that $(3)\Leftrightarrow (4)$. This finishes the proof of the corollary.
\end{proof}

In the next theorem, we establish a necessary and sufficient condition for the operators $W_{g, \psi}W_{f, \varphi}^*$ and $W_{g, \psi}^*W_{f, \varphi}$ to be invertible on $\mathcal F^2$.

\begin{thm}\label{I}
  Suppose that $f, g, \varphi$ and $\psi$ are entire functions such that $f(0)g(0)\neq 0$. Assume that $W_{g, \psi}$ and $W_{f, \varphi}$ are bounded  on the Fock space  $\mathcal{F}^2$. Then the following are equivalent:\\
  $(1)$ $W_{g, \psi}W_{f, \varphi}^*$ is invertible on $\mathcal{F}^2$;\vspace{2mm}\\
   $(2)$ $W_{g, \psi}^*W_{f, \varphi}$ is invertible on $\mathcal{F}^2$;\vspace{2mm}\\
   $(3)$ $|\varphi^{\prime}(z)\psi^{\prime}(z)|=1$ for all $z\in \mathbb C$.
\end{thm}

\begin{proof}
 Firstly, we  prove that $(1)\Leftrightarrow (3)$.
If $W_{g, \psi}W_{f, \varphi}^*$ is invertible, then there exists an invertible operator $T$ such that $W_{g, \psi}W_{f, \varphi}^*T=I$, where $I$ is the identity operator on $\mathcal F^2$. Thus $W_{g, \psi}$ is surjective. Then Theorem \ref{L7} tells us  that $W_{g,\psi}$ is invertible and so is $W_{f, \varphi}$. Using  Theorem \ref{L7} again, we conclude that $|\varphi^{\prime}(z)\psi^{\prime}(z)|=1$. Thus we get that $(1)\Rightarrow (3)$.

Using the assumption that $|\varphi^{\prime}(z)\psi^{\prime}(z)|=1$ for all $z\in \mathbb C$ and Theorem $\ref{L7}$, we obtain that $W_{f, \varphi}$ and $W_{g, \psi}$ are  both invertible, and so is $W_{g, \psi}W_{f, \varphi}^*$. This gives $(3)\Rightarrow (1)$.

Next we will prove that $(2)\Leftrightarrow (3)$. Suppose $W_{g, \psi}^*W_{f, \varphi}$ is invertible. Then  there exists a constant $\delta>0$ such that
$$\|W_{g, \psi}^*\|\|W_{f, \varphi}h\| \geqslant \|W_{g, \psi}^*W_{f, \varphi}h\|\geqslant \delta \|h\|$$
for any $h\in \mathcal{F}^2$.
It follows that
$$\|W_{f, \varphi}h\| \geqslant \frac{\delta \|h\|}{\|W_{g, \psi}^*\|}, \ \ h\in \mathcal{F}^2,$$
 to get that $W_{f, \varphi}$ is bounded below on $\mathcal F^2$. From the proof of Theorem \ref{U}, we have $|\varphi^{\prime}(z)|=1,$ which implies that $W_{f, \varphi}$ is invertible.  This shows that $W_{g, \psi}^*$ is invertible, and so is $W_{g, \psi}$. Thus we have $|\psi^{\prime}(z)|=1$ and  $(2)\Rightarrow (3)$. The proof of $(3)\Rightarrow (2)$ is similar to that of  $(3)\Rightarrow (1).$ This completes the proof of Theorem \ref{I}.
\end{proof}

Based on Theorems \ref{U} and \ref{I}, we show in the following theorem that the product  $W_{g, \psi}W_{f, \varphi}^*$ (or $W_{g, \psi}^*W_{f, \varphi}$) is unitary if and only if it is an isometry on the Fock space.

\begin{thm}\label{M5}
  Let $f, g$ and $\varphi, \psi$ be entire functions  such that $f(0)g(0)\neq 0$. Assume that $W_{g, \psi}$ and $W_{f, \varphi}$ are bounded weighted composition operators on $\mathcal{F}^2$. Then the following are equivalent:\\
   $(1)$ $W_{g, \psi}W_{f, \varphi}^*$ is isometric on $\mathcal{F}^2$;\\
   $(2)$ $W_{g, \psi}^*W_{f, \varphi}$ is isometric on $\mathcal{F}^2$;\\
   $(3)$ $W_{g, \psi}W_{f, \varphi}^*$ is unitary on $\mathcal{F}^2$;\\
    $(4)$ $W_{g, \psi}^*W_{f, \varphi}$ is unitary on $\mathcal{F}^2$;\\
    $(5)$  $|\varphi^{\prime}(z)\psi^{\prime}(z)|=1$ for all $z\in \mathbb C$ and  $|f(0)g(0)|e^{\frac{|\varphi(0)|^2+|\psi(0)|^2}{2}}=1$.
\end{thm}
\begin{proof}

   Let us show   $(2)\Rightarrow (4)$ first.  Assuming that $W_{g, \psi}^*W_{f, \varphi}$ is isometric on $\mathcal{F}^2$,  we are going to show $|\varphi^{\prime}(z)\psi^{\prime}(z)|=1$.  Denoting $A=W_{g, \psi}W_{g, \psi}^*$, we have
  $$I=(W_{g, \psi}^*W_{f, \varphi})^*W_{g, \psi}^*W_{f, \varphi}=W_{f, \varphi}^*AW_{f, \varphi}.$$
From the proof of Theorem \ref{U}, we have that $W_{f, \varphi}$ is bounded below.
 Since $|\varphi^{\prime}(z)|=0$ implies that $W_{f, \varphi}$ is compact, we  have $|\varphi^{\prime}(z)|=1$. Thus $W_{f, \varphi}$ and $W_{f, \varphi}^*$ are both invertible, and so is $W_{g, \psi}W_{g, \psi}^*$. By Theorem \ref{I}, we  obtain $|\psi^{\prime}(z)|=1$ immediately and hence $W_{g, \psi}$ is invertible.
    Consequently, the isometry $W_{g, \psi}^*W_{f, \varphi}$ is invertible and hence it is unitary.

Next we are going to show that $(1)\Rightarrow (3)$. Assuming $(1)$ holds,  we have $$W_{f, \varphi}W_{g, \psi}^*W_{g, \psi}W_{f, \varphi}^*=I,$$
 which yields that $W_{f, \varphi}$ is surjective, and hence $|\varphi^{\prime}(z)|=1.$ It follows that both $W_{f, \varphi}$ and $W_{f, \varphi}^*$ are invertible, so is $W_{g, \psi}^*W_{g, \psi}$.  Using Theorem \ref{I} again, we have that $|\psi^{\prime}(z)|=1$ and hence $W_{g, \psi}$ is invertible. Then $W_{g, \psi}W_{f, \varphi}^*$ is unitary and thus we establish that $(1)\Rightarrow (3)$. Therefore,  $(1)\Leftrightarrow (3)$ and $(2)\Leftrightarrow (4)$.  Combining this  with Theorem \ref{U}, we finish the proof of  Theorem \ref{M5}.
  \end{proof}

In the rest of this section, we turn to the discussion of the compactness of the products $W_{g, \psi}W_{f,\varphi}^*$  and $W_{g, \psi}^*W_{f,\varphi}$.

      \begin{prop} \label{T2}
    Suppose $W_{f, \varphi}$ and $W_{g, \psi}$ are two bounded weighted composition operators on $\mathcal{F}^2$ with $f(0)g(0)\neq 0$.  Then we have:\\
     $(1)$ if $W_{g, \psi}W_{f,\varphi}^*$ is compact on $\mathcal{F}^2$, then $|\varphi^{\prime}(z)\psi^{\prime}(z)|<1$;\vspace{2mm}\\
     $(2)$ if $|\varphi^{\prime}(z)\psi^{\prime}(z)|<1$ for all $z\in \mathbb C$, $\varphi(0)\psi(0)\neq 0$ and $\rho(f)+\rho(g)<3$, then $W_{g, \psi}W_{f,\varphi}^*$ is compact.
    \end{prop}
     \begin{proof}
  By Theorem~\ref{L1}, we can write
$$\varphi(z)=az+b,\qquad \psi(z)=cz+d,\qquad |a|,|c|\leq 1.$$  Let us prove (1) first. Suppose $W_{g, \psi}W_{f,\varphi}^*$ is compact on $\mathcal{F}^2$.
      If $|\varphi^{\prime}(z)\psi^{\prime}(z)|=1$, i.e., $|ac|=1$, then we have that $|a|=|c|=1$. By Theorem \ref{L7}, we have that both $W_{g, \psi}$ and $W_{f, \varphi}$ are invertible, so is $W_{g, \psi}W_{f, \varphi}^*$. This is a contradiction, since $W_{g, \psi}W_{f, \varphi}^*$ is compact. Thus $(1)$ holds.

        Now we are going to prove (2). Suppose that  $|\varphi^{\prime}(z)\psi^{\prime}(z)|=|ac|<1$ and $\rho(f)+\rho(g)<3$. To finish the proof, we consider the following three cases:

        (i) $|a|=1, |c|<1$. In this case, it was shown in Lemma \ref{|a|=1} that $$f(z)=f(0)e^{-a{\overline{b}z}}.$$
        We claim that $\rho(f)=1$. Indeed,
         since $\varphi(0)=b\neq 0$ and $|a|=1$, we can write $-a\overline{b}=|ab|e^{\mathrm{i}t}=|b|e^{\mathrm{i}t}$ for some $t\in \R$. Writing $z=re^{\mathrm{i}\theta}$ with $r>0$ and $\theta \in [-\pi, \pi]$, we obtain
         $$|f(z)|=\big|f(0)e^{-a{\overline{b}z}}\big|=|f(0)|e^{\mathrm{Re}(-a{\overline{b}z})}=|f(0)|e^{\mathrm{Re}[r|b|e^{\mathrm{i}(\theta+t)}]}=|f(0)|e^{r|b|\cos(\theta+t)}.$$
        It follows that
         $$M(r)=\sup_{|z|=r}|f(z)|=\sup_{\theta\in [-\pi, \pi]} |f(0)|e^{r|b|\cos(\theta+t)}=|f(0)|e^{r|b|},$$
    which yields
         \begin{align*}
           \begin{split}
           \rho(f)&=\limsup_{r\rightarrow \infty}  \frac{\log \log M(r)}{\log r}\\
           &=\limsup_{r\rightarrow \infty}\frac{\log (\log |f(0)|+r|b|)}{\log r}\\
           &=\limsup_{r\rightarrow \infty}\frac{r|b|}{\log |f(0)|+r|b|}\\
           &=1.
           \end{split}
         \end{align*}
         By the assumption, we deduce that $\rho(g)<2$. Then it follows  from Theorem $\ref{L3}$  that $W_{g, \psi}$ is compact, and hence $W_{g, \psi}W_{f, \varphi}^*$ is compact.

        (ii)  $|c|=1, |a|<1$. Applying the same argument as the one used in (i), we obtain that $W_{g, \psi}W_{f, \varphi}^*$ is compact.

        (iii) $|a|<1, |c|<1$. In this case, the desired conclusion follows from Theorem $\ref{L3}$ immediately.

This completes the proof of Proposition \ref{T2}.
    \end{proof}

Using the same method as the one used in the proof of Proposition \ref{T2}, we can also obtain  a necessary condition and a sufficient condition for $W_{g, \psi}^*W_{f,\varphi}$  to be compact on the Fock space.

   \begin{prop}\label{T3}
    Suppose that $W_{f, \varphi}$ and $W_{g, \psi}$ are two  bounded weighted composition operators on $\mathcal{F}^2$ with $f(0)g(0)\neq 0$. Then we have:\\
     $(1)$ if $W_{g, \psi}^*W_{f,\varphi}$ is compact on $\mathcal{F}^2$, then $|\varphi^{\prime}(z)\psi^{\prime}(z)|<1$;\vspace{2mm}\\
     $(2)$ if $|\varphi^{\prime}(z)\psi^{\prime}(z)|<1$ for all $z\in \mathbb C$, $\varphi(0)\psi(0)\neq 0$ and $\rho(f)+\rho(g)<3$, then $W_{g, \psi}^*W_{f,\varphi}$ is compact.
    \end{prop}

For unweighted composition operators, the following corollary tells us that the compactness of the products  $C_{\psi}C_{\varphi}^*$ and $C_{\psi}^*C_{\varphi}$ is equivalent.

\begin{cor}\label{C2}
   Suppose that $C_{\varphi}$ and $C_{\psi}$ are  bounded composition operators on the Fock space $\mathcal{F}^2$. Then the following are equivalent:\\
   $(1)$ $C_{\psi}C_{\varphi}^*$ is compact on $\mathcal{F}^2$;\\
   $(2)$ $C_{\psi}^*C_{\varphi}$ is compact on $\mathcal{F}^2$;\\
   $(3)$ $|\varphi^{\prime}(z)\psi^{\prime}(z)|<1$ for all $z\in \mathbb C$.
\end{cor}
\begin{proof}
   Note that $(1)\Rightarrow (3)$ and $(2)\Rightarrow (3)$ were obtained in  Propositions \ref{T2} and  \ref{T3}, respectively. It remains to show that $(3)\Rightarrow (1)$ and $(3)\Rightarrow (2)$.

   Assume that (3) holds. Then  at least one of $|\varphi^{\prime}(z)|$ and $|\psi^{\prime}(z)|$ is less than $1$ for all $z\in \mathbb C$.
    According to Theorem $\ref{L3}$, we deduce that one of $C_{\varphi}$ and $C_\psi$ is compact, which gives $(1)$ and $(2)$, finishing the proof.
\end{proof}

\section{Products with one bounded factor}\label{S4}

 In this section, we study the boundedness, compactness, invertibility, and isometry of the
 product $W_{g,\psi}W_{f,\varphi}^*$ under the  assumption that one of the two factors is bounded and the other is unbounded.

Let us begin with a necessary condition for  the boundedness of $W_{g,\psi}W_{f,\varphi}^*$.

\begin{thm}\label{ext-affine}
Suppose that $f$ and $g$ are entire functions with $f(0)g(0)\neq0$. Let $W_{f,\varphi}$ and $W_{g,\psi}$ be densely defined maximal weighted composition operators on $\mathcal F^2$. Assume that exactly one of $W_{f,\varphi}$ and $W_{g,\psi}$ is bounded and that the natural product $T_0=W_{g,\psi}W_{f,\varphi}^*$ admits a unique bounded extension $T$ to $\mathcal F^2$. Then $f,g\in\mathcal F^2$, and there exist constants $a,b,c,d\in\C$ such that
$$\varphi(z)=az+b,\qquad \psi(z)=cz+d,\qquad 0<|ac|\leqslant1.$$
Moreover, if $W_{f,\varphi}$ is bounded, then $|a|\leqslant1$; if $W_{g,\psi}$ is bounded, then $|c|\leqslant1$.
\end{thm}

\begin{proof}
We first prove that $f,g\in\mathcal F^2$.
Suppose that $W_{f,\varphi}$ is bounded. Then $f=W_{f,\varphi}1\in\mathcal F^2$,
and Theorem \ref{L1} gives $\varphi(z)=az+b$ with $|a|\leqslant 1$. Since $\varphi$ is nonconstant, $a\neq0$. Let $w_0=-\frac{b}{a}$. If $f(w_0)=0$,
since the zeros of $f$ are isolated, we can find an $r>0$ such that $f(w)\neq0$ in $0<|w-w_0|\leqslant r$; if $f(w_0)\neq 0$, by the continuity of $f$, we can also choose an $r>0$ such that $f(w)\neq0$ in $0<|w-w_0|\leqslant r$. Thus  we can choose $w_1=w_0+\frac{r}{2}, w_2=w_0-\frac{r}{2}$ such that $f(w_1)f(w_2)\neq0$. Moreover, $w_1+w_2=2w_0=-\frac{2b}{a}$. Thus
$$
\varphi(w_1)+\varphi(w_2)
=a(w_1+w_2)+2b=0.
$$
Using \eqref{ext-kernel} in Lemma \ref{ext-domain}, we obtain
$$
(TK_{w_1})(z)(TK_{w_2})(z)
=\overline{f(w_1)f(w_2)}g(z)^2.
$$
It follows that
\begin{align*}
  \int_{\C}|g(z)|^2d\mu(z)&=\int_\C \frac{\big|(TK_{w_1})(z)(TK_{w_2})(z)\big|}{|f(w_1)f(w_2)|}d\mu(z)\\
  &\leqslant\frac{\|TK_{w_1}\|\|TK_{w_2}\|}{|f(w_1)f(w_2)|}<\infty
\end{align*}
Hence $g\in\mathcal F^2$.

If $W_{g,\psi}$ is bounded, we can use the same argument to derive $\psi(z)=cz+d$ with $c\neq 0$ and then to obtain $f, g\in \mathcal F^2$.

Next we prove $\varphi$ and $\psi$ are affine under the assumption that $T$ is a bounded extension.
By Lemma \ref{ext-domain} and the reproducing property, for all $z,w\in\C$,
\begin{align*}
|f(w)g(z)|e^{\mathrm{Re}(\overline{\varphi(w)}\psi(z))}=|TK_w(z)|=|\langle TK_w,K_z\rangle|\leqslant\|T\|\|K_w\|\|K_z\|=\|T\|e^{\frac{|w|^2+|z|^2}{2}}.
\end{align*}
Consequently,
\begin{align}\label{ext-two-variable}
|f(w)g(z)|e^{\mathrm{Re}(\overline{\varphi(w)}\psi(z))
-\frac{|w|^2+|z|^2}{2}}
\leqslant\|T\|,\qquad z,w\in\C.
\end{align}

 Define the following entire functions by
$$
u(z)=\overline{f(\overline{z})}g(z),\qquad
v(z)=u(z)e^{\overline{\varphi(\overline{z})}\psi(z)}.$$
 Since $f,g\in\mathcal F^2$, the point evaluation estimate gives
$$|u(z)|\leqslant\|f\|\|g\|e^{|z|^2},\qquad z\in\C.$$
Taking $w=\overline{z}$ in \eqref{ext-two-variable}, we also obtain
$$|v(z)|\leqslant\|T\|e^{|z|^2},\qquad z\in\C.$$
Thus $u$ and $v$ are nonzero entire functions of order at most $2$.
Since
$$e^{\overline{\varphi(\overline{z})}\psi(z)}\neq0,\qquad z\in\C,$$
$u$ and $v$ have the same zeros, counted with multiplicities.
Let $\{\xi_j\}$ denote the common zeros of $u$ and $v$,
counted with multiplicities. Since $u(0)v(0)\neq0$, all
$\xi_j$ are nonzero.
By Lemma \ref{zero-summability},
$$\sum_{j=1}^{\infty}\frac{1}{|\xi_j|^3}<\infty.$$

Next we prove  that the canonical product
$$\Pi(z)=\prod_{j=1}^{\infty}E_2\left(\frac{z}{\xi_j}\right)$$
converges locally uniformly in $\C$.
If the common zero set of $u$ and $v$ is finite, the assertion
is immediate. Otherwise, fix $R>0$ and let
$$\Omega=\{z\in\C:|z|<R\}.$$
Since $$\sum_{j=1}^{\infty}\frac{1}{|\xi_j|^3}<\infty,$$
 we have
$$
\lim_{j\rightarrow\infty}\frac{1}{|\xi_j|^3}=0.
$$
Hence there exists a positive integer $J$ such that
$$
\frac{1}{|\xi_j|^3}\leqslant\frac{1}{(2R)^3},
\qquad j\geqslant J.
$$
Consequently, $|\xi_j|\geqslant2R$ for all $j\geqslant J$,
and therefore
$$
\left|\frac{z}{\xi_j}\right|<\frac{1}{2},
\qquad z\in\Omega,\quad j\geqslant J.
$$
By Lemma \ref{lem:canonical-factor-estimate}, there exists a
constant $C_1>0$ such that
$$
\left|E_2\left(\frac{z}{\xi_j}\right)-1\right|
\leqslant C_1\left|\frac{z}{\xi_j}\right|^3
\leqslant\frac{C_1R^3}{|\xi_j|^3},
\qquad z\in\Omega,\quad j\geqslant J.
$$
For $1\leqslant j\leqslant J-1$, since the finitely many functions
$E_2(\frac{z}{\xi_j})-1$ are continuous
on the compact set $\overline{\Omega}$, they are bounded
on $\Omega$. Thus, there exists a constant $C_2>0$ such that
$$\left|E_2\left(\frac{z}{\xi_j}\right)-1\right|\leqslant \frac{C_2R^3}{|\xi_j|^3}.$$
Taking $C=\max\{C_1, C_2\}$, we have $$\left|E_2\left(\frac{z}{\xi_j}\right)-1\right|\leqslant \frac{CR^3}{|\xi_j|^3},\ \ \ z\in \Omega,\ \  j\geqslant 1.$$
On the other hand, since
$$
\sum_{j=J}^{\infty}\frac{CR^3}{|\xi_j|^3}
\leqslant
CR^3\sum_{j=1}^{\infty}\frac{1}{|\xi_j|^3}
<\infty,
$$
we have
$$\sum_{j=1}^{\infty}\frac{CR^3}{|\xi_j|^3}<\infty.$$
It follows from
Lemma \ref{lem:holomorphic} that
$$
\prod_{j=1}^{\infty}E_2\left(\frac{z}{\xi_j}\right)
$$
converges uniformly on $\Omega$.  Since $R>0$ was arbitrary, $\Pi$ converges locally uniformly in $\C$ and defines an entire function.

Fix $h\in\{u,v\}$, and let $\rho(h)$ denote the order of $h$.
Choose the integer $k$ such that
$k\leqslant\rho(h)<k+1$. Since $\rho(h)\leqslant2$, we have
$k\in\{0,1,2\}$. By Theorem \ref{hadamard-factorization},
there exists a polynomial $P$ of degree at most $k$ such that
$$
h(z)=e^{P(z)}
\prod_{j=1}^\infty E_k\left(\frac{z}{\xi_j}\right),
\qquad z\in\C.
$$
Note that $h(0)\neq0$, since $u(0)v(0)\neq0$. We consider three cases.

\textbf{Case 1}. If $k=2$, this already gives $h=\Pi e^P$ with
$\deg P\leqslant2$.

\textbf{Case 2}. If $k=1$, then $\rho(h)<2$. Applying Lemma \ref{zero-summability}, we obtain
$$
\sum_{j=1}^\infty\frac{1}{|\xi_j|^2}<\infty.
$$
Since
$$E_2(\zeta)=(1-\zeta)e^{\zeta+\frac{\zeta^2}{2}}=E_1(\zeta)e^{\frac{\zeta^2}{2}},$$
we have
$$
E_1(\zeta)=E_2(\zeta)e^{-\frac{\zeta^2}{2}},\ \ \ \zeta\in \C.
$$
For every integer $N\geqslant1$, substituting $\zeta=\frac{z}{\xi_j}$ into the above equality gives
\begin{align}\label{eq:k=1}
\prod_{j=1}^{N}E_1\left(\frac{z}{\xi_j}\right)
&=
\left[\prod_{j=1}^{N}E_2\left(\frac{z}{\xi_j}\right)\right]
\prod_{j=1}^{N}e^{-\frac{z^2}{2\xi_j^2}}=
\left[\prod_{j=1}^{N}E_2\left(\frac{z}{\xi_j}\right)\right]
\exp\left\{-\frac{z^2}{2}
\sum_{j=1}^{N}\frac{1}{\xi_j^2}\right\}.
\end{align}
Since
$$\sum_{j=1}^{\infty}\left|\frac{1}{\xi_j^2}\right|
=
\sum_{j=1}^{\infty}\frac{1}{|\xi_j|^2}<\infty,$$
taking the limit in \eqref{eq:k=1} gives
$$
h(z)=e^{P(z)}\lim_{N\rightarrow\infty}\prod_{j=1}^{N}E_1\left(\frac{z}{\xi_j}\right)
=\Pi(z)\exp\left\{
P(z)-\frac{z^2}{2}\sum_{j=1}^{\infty}\frac{1}{\xi_j^2}
\right\}.
$$
Here $\deg P\leqslant1$, so the polynomial in the exponent
has degree at most $2$.

\textbf{Case 3}. If $k=0$, then $\rho(h)<1$. Applying  Lemma \ref{zero-summability} again, we have
$$
\sum_{j=1}^{\infty}\frac{1}{|\xi_j|}<\infty\ \ \ \text{and}\ \ \
\sum_{j=1}^{\infty}\frac{1}{|\xi_j|^2}<\infty.
$$
Using
$$
E_0(\zeta)=E_2(\zeta)e^{-\zeta-\frac{\zeta^2}{2}},
$$
we similarly obtain
$$
h(z)=\Pi(z)\exp\left\{
P(z)-z\sum_{j=1}^{\infty}\frac{1}{\xi_j}
-\frac{z^2}{2}\sum_{j=1}^{\infty}\frac{1}{\xi_j^2}
\right\}.
$$
Since $P$ is constant, the polynomial in this exponent
also has degree at most $2$.

Applying these conclusions to $u$ and $v$, respectively,
we obtain polynomials $p_1,p_2$ of degree at most $2$
such that
$$
u(z)=\Pi(z)e^{p_1(z)},\qquad
v(z)=\Pi(z)e^{p_2(z)},\qquad z\in\C.
$$
Setting $p=p_2-p_1$, we have $\deg p\leqslant2$ and
$$
v(z)=\Pi(z)e^{p_2(z)}
=u(z)e^{p(z)},\qquad z\in\C.
$$
On the other hand, the definition of $v$ gives
$$
v(z)=u(z)e^{\overline{\varphi(\overline{z})}\psi(z)},\qquad z\in\C.
$$
Hence
$$
e^{\overline{\varphi(\overline{z})}\psi(z)}=e^{p(z)}
$$
whenever $u(z)\neq0$. Since $u$ is not identically zero, the identity theorem
extends this equality to all of $\C$. Therefore,
$$
e^{\overline{\varphi(\overline{z})}\psi(z)-p(z)}=1,\qquad z\in\C.
$$
Differentiating the preceding identity with respect to $z$, we obtain
$$
\frac{d}{dz}\left(
\overline{\varphi(\overline{z})}\psi(z)-p(z)
\right)=0.
$$
Hence there exists a constant $M\in\C$ such that
$$
\overline{\varphi(\overline{z})}\psi(z)-p(z)=M.
$$
It follows that
$$\overline{\varphi(\overline{z})}\psi(z)=p(z)+M:=q(z),\ \ \ z\in \C,$$
where $q(z)$ is a polynomial with degree at most 2.

Suppose first that $W_{f,\varphi}$ is bounded. By Theorem \ref{L1}$(1)$ and the assumption that $\varphi$ is nonconstant, we have $\varphi(z)=az+b$ for some constants $a,b\in\C$ with $a\neq0$. Therefore,
$$(\overline{a}z+\overline{b})\psi(z)=q(z).$$
Since $\psi$ is entire, the above equality implies  $\psi$ must be  a polynomial of degree at most $1$. As $\psi$ is nonconstant, it follows that $\psi(z)=cz+d$ with $c\neq0$.
Similarly, if $W_{g,\psi}$ is bounded, we obtain the desired affine forms of $\varphi$ and $\psi$.

We next prove that $|ac|\leqslant1$. Suppose $|ac|>1$. Taking $z=\frac{a\overline{c}}{|ac|}w$ in \eqref{ext-two-variable}, we obtain
\begin{align*}
&\left|
f(w)g\left(\frac{a\overline{c}}{|ac|}w\right)
e^{\left(a\overline{d}
+\frac{a|c|^2\overline{b}}{|ac|}\right)w}
\right|\leqslant
\|T\|e^{-\mathrm{Re}(\overline{b}d)}
e^{(1-|ac|)|w|^2},
\qquad w\in\C.
\end{align*}
Thus the entire function inside the modulus is bounded.
By Liouville's theorem, it is constant, and its value at
$w=0$ is $f(0)g(0)$. It follows that
$$f(w)g\left(\frac{a\overline{c}}{|ac|}w\right)
e^{\left(a\overline{d}
+\frac{a|c|^2\overline{b}}{|ac|}\right)w}=f(0)g(0),\ \ \ w\in\C.$$
So
$$|f(0)g(0)|\leqslant \|T\|e^{-\mathrm{Re}(\overline{b}d)}
e^{(1-|ac|)|w|^2},
\qquad w\in\C.$$
Taking $w=n\in\N$ and letting $n\to \infty$, we obtain
$$
0<|f(0)g(0)|
\leqslant
\lim_{n\rightarrow\infty}
\|T\|e^{-\mathrm{Re}(\overline{b}d)}
e^{(1-|ac|)n^2}
=0,
$$
since $|ac|>1$. This contradicts $f(0)g(0)\neq0$. Hence $|ac|\leqslant1$.

Finally, if $W_{f,\varphi}$ is bounded, then (1) of Theorem \ref{L1} gives $|a|\leqslant1$. Similarly, if $W_{g,\psi}$ is bounded, then $|c|\leqslant1$. This completes the proof.
\end{proof}

In the following results, membership of the weights in $\mathcal F^2$ and the affine forms of the symbols follow from Theorem \ref{ext-affine}.

\begin{prop}\label{ext-boundary}
Suppose that $f$ and $g$ are entire functions with $f(0)g(0)\neq0$.
Assume that the natural product $W_{g,\psi}W_{f,\varphi}^*$ has a unique
bounded extension $T$ to $\mathcal F^2$, and that exactly one of
$W_{f,\varphi}$ and $W_{g,\psi}$ is bounded on $\mathcal F^2$.
If $T$ is compact, then
$$|\varphi^{\prime}(z)\psi^{\prime}(z)|<1,\qquad z\in\C.$$
\end{prop}
\begin{proof}
Suppose that $T$ is compact.
By Theorem \ref{ext-affine}, there exist $a,b,c,d\in\C$ such that
$$\varphi(z)=az+b,\qquad\psi(z)=cz+d,\qquad0<|ac|\leqslant1.$$

Suppose that $|ac|=1$. Taking $z=a\overline{c} w$ in
\eqref{ext-two-variable}, we obtain
$$\left|f(w)g(a\overline{c} w)
e^{(a\overline{d}+a|c|^2\overline{b})w}\right|
\leqslant\|T\|e^{-\mathrm{Re}(\overline{b} d)},\qquad w\in\C.$$
As in the proof of Theorem \ref{ext-affine}, we have
\begin{align}\label{eq:constant}
f(w)g(a\overline{c} w)e^{(a\overline{d}+a|c|^2\overline{b})w}
=f(0)g(0),\qquad w\in\C.
\end{align}
Using  the kernel formula \eqref{ext-kernel} in Lemma \ref{ext-domain}
and $|a\overline{c} w|=|w|$, we can derive that
\begin{align}\label{eq:boundary-kernels}
\left|\left\langle Tk_w,k_{a\overline{c} w}\right\rangle\right|
&=e^{-|w|^2}\left|TK_w(a\overline{c} w)\right|\notag\\
&=e^{\mathrm{Re}(\overline{b} d)}
\left|f(w)g(a\overline{c} w)e^{(a\overline{d}+a|c|^2\overline{b})w}\right|\notag\\
&=e^{\mathrm{Re}(\overline{b} d)}|f(0)g(0)|,
\qquad w\in\C,
\end{align}
where the last equality follows from \eqref{eq:constant}.

Note that the sequence of normalized reproducing kernels $\{k_n\}$ converges
weakly to zero in $\mathcal F^2$.
Compactness of $T$ therefore implies
$$\lim_{n\rightarrow\infty}\|Tk_n\|=0.$$
Taking $w=n$ in \eqref{eq:boundary-kernels} and using
the Cauchy--Schwarz inequality together with
$\|k_{a\overline{c} n}\|=1$, we obtain
\begin{align*}
0<e^{\mathrm{Re}(\overline{b} d)}|f(0)g(0)|
&=\lim_{n\rightarrow\infty}
\left|\left\langle Tk_n,k_{a\overline{c} n}\right\rangle\right|\leqslant\lim_{n\rightarrow\infty}\|Tk_n\|=0,
\end{align*}
a contradiction. Hence $|ac|\neq1$, and consequently
$$|\varphi^{\prime}(z)\psi^{\prime}(z)|=|ac|<1,\qquad z\in\C.$$
This finishes the proof.
\end{proof}

We next characterize isometric and unitary extensions. To begin with, we first give  the Gaussian integral formula (see \cite[Eq.~(2)]{DZ}) that will be used repeatedly in the following:
for $\theta,p,q\in\C$ with $|\theta|<\frac{1}{2}$,
\begin{align}\label{ext-Gaussian-integral}
\int_{\C}e^{\theta z^2+\overline{\theta}\overline{z}^2+pz+q\overline{z}}d\mu(z)
=\frac{1}{\sqrt{1-4|\theta|^2}}
\exp\left\{\frac{pq+\overline{\theta}p^2+\theta q^2}{1-4|\theta|^2}\right\}.
\end{align}
\begin{thm}\label{ext-unitary}
Suppose that $f$ and $g$ are entire functions with $f(0)g(0)\neq0$.
Assume that the natural product $W_{g,\psi}W_{f,\varphi}^*$ has a unique
bounded extension $T$ to $\mathcal F^2$, and that exactly one of
$W_{f,\varphi}$ and $W_{g,\psi}$ is bounded on $\mathcal F^2$.
Then the following assertions are equivalent:\\
$(1)$ $T$ is isometric.\\
$(2)$ $T$ is unitary.\\
$(3)$ There exist $\alpha,\beta,\gamma,\delta\in\C$ with
$|\alpha|,|\gamma|<\frac{1}{2}$ such that
$$f(z)=f(0)e^{\alpha z^2+\beta z},\qquad
g(z)=g(0)e^{\gamma z^2+\delta z}.$$
Moreover, with $\Delta=1-4|\gamma|^2$, the following identities hold for every $z\in\C$:
\begin{equation}\label{ext-unitary-conditions}
\begin{cases}
\Delta=|\varphi^{\prime}(z)\psi^{\prime}(z)|^2>0,\vspace{2mm}\\
\Delta\overline{\alpha}
+\overline{\gamma}\big(\overline{\varphi^{\prime}(z)}\psi^{\prime}(z)\big)^2=0,\vspace{2mm}\\
\Delta\big(\overline{\beta}+\overline{\varphi^{\prime}(z)}\psi(0)\big)+\overline{\varphi^{\prime}(z)}\psi^{\prime}(z)
\Big[\overline{\delta}+\overline{\psi^{\prime}(z)}\varphi(0)
+2\overline{\gamma}\big(\delta+\psi^{\prime}(z)\overline{\varphi(0)}\big)\Big]=0,
\end{cases}
\end{equation}
and
\begin{align}\label{ext-normalization}
\begin{split}
&\frac{|f(0)g(0)|^2
e^{2\mathrm{Re}[\overline{\varphi(0)}\psi(0)]}}{\sqrt{\Delta}}\times\exp\left\{
\frac{|\delta+\psi^{\prime}(z)\overline{\varphi(0)}|^2
+2\mathrm{Re}\big[\overline{\gamma}
\big(\delta+\psi^{\prime}(z)\overline{\varphi(0)}\big)^2\big]}{\Delta}
\right\}=1.
\end{split}
\end{align}
\end{thm}
\begin{proof}
By Theorem \ref{ext-affine}, we have that $f,g\in\mathcal F^2$ and there exist $a,b,c,d\in\C$ such that
$$\varphi(z)=az+b,\qquad\psi(z)=cz+d,\qquad0<|ac|\leqslant1.$$
Thus $\varphi^{\prime}(z)=a$, $\psi^{\prime}(z)=c$, $\varphi(0)=b$ and $\psi(0)=d$.
We prove $(1)\Rightarrow(3)\Rightarrow(2)\Rightarrow(1)$.

Suppose first that $T$ is isometric. If $f(w)=0$, then
\eqref{ext-kernel} in Lemma \ref{ext-domain} gives $TK_w=0$, a contradiction.
Hence $f$ has no zeros. By Theorem \ref{hadamard-factorization}, we can write $f$ as
$$f(z)=e^{\alpha z^2+\beta z+\eta}$$ for some $\alpha, \beta, \eta\in \C$.
Since $f(0)=e^{\eta}$, applying Lemma \ref{ext-exponential-fock} to $f$ yields
$$f(z)=f(0)e^{\alpha z^2+\beta z},\qquad |\alpha|<\frac{1}{2}.$$

We next prove that $g$ has no zeros by using
$\|TK_w\|^2=\|K_w\|^2=e^{|w|^2}$.
Let
$$H(z)=\frac{|f(0)|^2}{\pi}|g(z)|^2
e^{2\mathrm{Re}(\overline{b}d+c\overline{b}z)-|z|^2}\geqslant0.$$
Substituting the expansion of $\varphi, \psi, f$  and $H$ into \eqref{ext-kernel} gives

\begin{align*}
\|TK_w\|^2&=|e^{\alpha w^2+\beta w}|^2e^{2\mathrm{Re}(\overline a d\overline w)}\int_{\mathbb C}H(z)e^{2\mathrm{Re}(\overline a c\overline w z)}\,dA(z)\\
&=e^{2\mathrm{Re}[\alpha w^2+(\beta+a\overline d)w]}\int_{\mathbb C}H(z)e^{2\mathrm{Re}(\overline a c\overline w z)}\,dA(z).
\end{align*}
Thus we have
\begin{align}\label{ext-H-integral}
\int_{\C}H(z)e^{2\mathrm{Re}(\overline{a}c\overline{w}z)}dA(z)
=\exp\left\{|w|^2-2\mathrm{Re}
\big[\alpha w^2+(\beta+a\overline{d})w\big]\right\},\qquad w\in\C.
\end{align}
We will identify $H$ with a strictly positive Gaussian function.
Put $D=1-4|\alpha|^2>0$. Taking
$\theta=-\alpha$, $p=\sqrt D\,\overline w$ and $q=\sqrt D\,w$
in \eqref{ext-Gaussian-integral}, and then setting $\zeta=\sqrt D\,z$
with $dA(\zeta)=D\,dA(z)$, gives
\begin{align}\label{ext-centered-Gaussian}
\frac{1}{\pi\sqrt D}\int_{\C}
\exp\left\{-\frac{|\zeta|^2+2\mathrm{Re}(\alpha\zeta^2)}{D}
+2\mathrm{Re}(\overline w\zeta)\right\}dA(\zeta)
=e^{|w|^2-2\mathrm{Re}(\alpha w^2)},\ \ w\in \C.
\end{align}
We  make the change of variables $\zeta=\overline{a}cz+\overline{\beta}+\overline{a}d$ in \eqref{ext-centered-Gaussian} so that $dA(\zeta)=|ac|^2dA(z)$. Define
\begin{align*}
G(z)=\frac{|ac|^2}{\pi\sqrt D}\exp\bigg\{
-\frac{|\overline a cz+\overline\beta+\overline a d|^2}{D}-\frac{2\mathrm{Re}\big[\alpha(\overline a cz+\overline\beta+\overline a d)^2\big]}{D}\bigg\}>0.
\end{align*}
Then \eqref{ext-centered-Gaussian}  yields
\begin{align}\label{ext-G-integral}
\int_{\C}G(z)e^{2\mathrm{Re}(\overline a c\overline w z)}dA(z)
=\exp\left\{|w|^2-2\mathrm{Re}
\big[\alpha w^2+(\beta+a\overline d)w\big]\right\},\ \ w\in \C.
\end{align}
Taking $w=0$ in \eqref{ext-H-integral} and \eqref{ext-G-integral} gives $$\int_{\C}H(z)\,dA(z)=\int_{\C}G(z)\,dA(z)=1.$$
For $z=x+\mathrm{i}y$ and $w=\frac{s+\mathrm{i}t}{2a\overline c}$, subtracting \eqref{ext-G-integral} from \eqref{ext-H-integral},
we obtain
$$\int_{\R^2}[H(x+\mathrm{i}y)-G(x+\mathrm{i}y)]e^{sx+ty}dx\,dy=0,
\qquad s,t\in\R.$$
For $s,t\in\mathbb R$, adding \eqref{ext-H-integral}
and \eqref{ext-G-integral} with
$w=\frac{s+\mathrm{i}t}{2a\overline c}$ gives 
\begin{align}\label{estimate:integrable}
\int_{\R^2}[H(x+\mathrm{i}y)+G(x+\mathrm{i}y)]
e^{sx+ty}\,dx\,dy<\infty,
\qquad s,t\in\R.
\end{align}
 Since $H\geqslant0$ and $G>0$, the above integral is absolutely convergent.

Now let $\lambda,\nu\in\C$ and choose $R>0$ such that
$|\lambda|,|\nu|\leqslant R$. Then
$$
|e^{\lambda x+\nu y}|
=e^{\mathrm{Re}(\lambda)x+\mathrm{Re}(\nu)y}
\leqslant e^{R(|x|+|y|)}.
$$
This gives
\begin{align*}
&\int_{\R^2}
\left|
[H(x+\mathrm{i}y)-G(x+\mathrm{i}y)]
e^{\lambda x+\nu y}
\right|\,dx\,dy\\
&\quad\leqslant
\int_{\R^2}[H(x+\mathrm{i}y)+G(x+\mathrm{i}y)]
e^{R(|x|+|y|)}\,dx\,dy.
\end{align*}
Since
$$
e^{R|x|}\leqslant e^{Rx}+e^{-Rx},
\qquad
e^{R|y|}\leqslant e^{Ry}+e^{-Ry},
$$
we have
\begin{align*}
e^{R(|x|+|y|)}
&\leqslant(e^{Rx}+e^{-Rx})(e^{Ry}+e^{-Ry})\\
&=e^{Rx+Ry}+e^{Rx-Ry}+e^{-Rx+Ry}+e^{-Rx-Ry}.
\end{align*}
Thus \eqref{estimate:integrable} yields
$$\int_{\R^2}
\left|[H(x+\mathrm{i}y)-G(x+\mathrm{i}y)]e^{\lambda x+\nu y}\right|\,dx\,dy<\infty.$$

Define
$$
F(\lambda,\nu)=\int_{\R^2}
[H(x+\mathrm{i}y)-G(x+\mathrm{i}y)]
e^{\lambda x+\nu y}\,dx\,dy,
\qquad \lambda,\nu\in\C.
$$
We have already proved that $F(s,t)=0$ for all $s,t\in\R$.
We next prove that this equality holds for all complex parameters.

We first show that $F$ is entire in each parameter separately.
The estimates above also show that for every $R>0$,
$$
\int_{\R^2}[H(x+\mathrm{i}y)+G(x+\mathrm{i}y)]
e^{R(|x|+|y|)}\,dx\,dy<\infty.
$$
Fix $|\lambda|,|\nu|\leqslant R$.  Since
$$|x|\leqslant e^{|x|+|y|},\ \ \ |y|\leqslant e^{|x|+|y|},$$
we have
\begin{align*}
\begin{cases}
\left|x[H(x+\mathrm{i}y)-G(x+\mathrm{i}y)]
e^{\lambda x+\nu y}\right|
\leqslant[H(x+\mathrm{i}y)+G(x+\mathrm{i}y)]
e^{(R+1)(|x|+|y|)},\vspace{2mm}\\
\left|y[H(x+\mathrm{i}y)-G(x+\mathrm{i}y)]
e^{\lambda x+\nu y}\right|
\leqslant[H(x+\mathrm{i}y)+G(x+\mathrm{i}y)]
e^{(R+1)(|x|+|y|)}.
\end{cases}
\end{align*}
The function on the right is independent of $\lambda,\nu$ and is
integrable. Thus we have
\begin{align*}
\begin{cases}
\frac{\partial F}{\partial\lambda}(\lambda,\nu)
=\int_{\R^2}x[H(x+\mathrm{i}y)-G(x+\mathrm{i}y)]
e^{\lambda x+\nu y}\,dx\,dy,\vspace{2mm}\\
\frac{\partial F}{\partial\nu}(\lambda,\nu)
=\int_{\R^2}y[H(x+\mathrm{i}y)-G(x+\mathrm{i}y)]
e^{\lambda x+\nu y}\,dx\,dy.
\end{cases}
\end{align*}
Since $R>0$ is arbitrary, $F$ is entire in $\lambda$ for each
fixed $\nu\in\C$, and entire in $\nu$ for each fixed $\lambda\in\C$.

Now fix $t\in\R$. The entire function
$\lambda\mapsto F(\lambda,t)$ vanishes for every real $\lambda$,
because $F(s,t)=0$ for all $s\in\R$. The identity theorem gives
$$
F(\lambda,t)=0,\qquad \lambda\in\C,\quad t\in\R.
$$
Next, fix $\lambda\in\C$. Similarly, we obtain
$$
F(\lambda,\nu)=0,\qquad \lambda,\nu\in\C.
$$

Taking $\lambda=-\mathrm{i}s$ and $\nu=-\mathrm{i}t$ in the equality $F(\lambda,\nu)=0$, we obtain
$$
\int_{\R^2}[H(x+\mathrm{i}y)-G(x+\mathrm{i}y)]
e^{-\mathrm{i}(sx+ty)}\,dx\,dy=0,
\qquad s,t\in\R.
$$
This says that the Fourier transform of
$ H(x+\mathrm{i}y)-G(x+\mathrm{i}y)$ vanishes everywhere.
Moreover, the integrals of $H$ and $G$ are both equal to $1$, so
$$
\int_{\R^2}|H(x+\mathrm{i}y)-G(x+\mathrm{i}y)|\,dx\,dy
\leqslant\int_{\C}[H(z)+G(z)]\,dA(z)=2.
$$
Thus this function belongs to $L^1(\R^2)$, and the uniqueness
theorem for the Fourier transform gives $H=G$ almost everywhere. Since $H-G$ is continuous, we have
$$
H(z)=G(z),\qquad z\in\C.
$$
Since $G>0$, the definition of $H$ implies that $g$ has no zeros.
Thus we also have
$$g(z)=g(0)e^{\gamma z^2+\delta z},\qquad |\gamma|<\frac{1}{2}.$$

Now, \eqref{ext-kernel} becomes
\begin{align}\label{ext-Gaussian-kernel}
\begin{split}
TK_w(z)=\overline{f(0)}g(0)\exp\big\{&\overline b d+\gamma z^2
+(\delta+c\overline b)z+\overline\alpha\overline w^2+(\overline\beta+\overline a d)\overline w
+\overline a cz\overline w\big\}.
\end{split}
\end{align}
Let $\Delta=1-4|\gamma|^2>0$. For $w,\omega\in\C$, denote
$$p=\delta+c\overline b+\overline a c\overline w,\qquad
q=\overline\delta+\overline c b+a\overline c\omega.$$
Applying \eqref{ext-Gaussian-integral} with $\theta=\gamma$ yields
\begin{align}\label{ext-kernel-inner-product}
\begin{split}
\langle TK_w,TK_\omega\rangle
={}&\frac{|f(0)g(0)|^2e^{2\mathrm{Re}(\overline b d)}}{\sqrt\Delta}
\exp\left\{\begin{aligned}
&\overline\alpha\overline w^2+\alpha\omega^2
+(\overline\beta+\overline a d)\overline w\\
&+(\beta+a\overline d)\omega
+\frac{pq+\overline\gamma p^2+\gamma q^2}{\Delta}
\end{aligned}\right\}.
\end{split}
\end{align}
Since
$$
\langle TK_w,TK_\omega\rangle
=\langle K_w,K_\omega\rangle=e^{\overline w\omega},
$$
the right hand side of \eqref{ext-kernel-inner-product} equals
$e^{\overline w\omega}$. Letting $w=\omega=0$ gives
\eqref{ext-normalization}.

Next,  differentiating \eqref{ext-kernel-inner-product} with
respect to $\overline w$ yields
$$
\omega e^{\overline w\omega}
=
e^{\overline w\omega}
\left[
2\overline\alpha\,\overline w
+\overline\beta+\overline a d
+\frac{\overline a c\,(q+2\overline\gamma p)}{\Delta}
\right],
$$
to obtain
$$
2\overline\alpha\overline w+\overline\beta+\overline a d
+\frac{\overline a c\,(q+2\overline\gamma p)}{\Delta}
=\omega.
$$
Substituting the expressions for $p$ and $q$ and comparing
the constant terms and the coefficients of $\overline w$ and
$\omega$ on both sides, we obtain, respectively,
\begin{align*}
\overline\beta+\overline a d
+\frac{\overline a c[\overline\delta+\overline c b
+2\overline\gamma(\delta+c\overline b)]}{\Delta}=0,\\
\overline\alpha+\frac{\overline\gamma(\overline a c)^2}{\Delta}=0\ \ \ \text{and}\ \ \
\frac{|ac|^2}{\Delta}=1.
\end{align*}
Multiplying these equalities by $\Delta$ and recalling that
$\varphi^{\prime}(z)=a$, $\psi^{\prime}(z)=c$,
$\varphi(0)=b$ and $\psi(0)=d$, we obtain
\eqref{ext-unitary-conditions}. This proves $(1)\Rightarrow(3)$.

Suppose next that $(3)$ holds. Substituting
\eqref{ext-unitary-conditions} and \eqref{ext-normalization} into
\eqref{ext-kernel-inner-product} implies that
$$\langle TK_w,TK_\omega\rangle=\langle K_w,K_\omega\rangle,
\qquad w,\omega\in\C.$$
Since $E:=\operatorname{span}\{K_w:\ w\in \C\}$ is dense in $\mathcal F^2$ and $T$ is
bounded, $T$ is isometric and $T^*T=I$.
It remains to prove that $T$ is onto.

Write
$$\lambda=\overline a c,\qquad u=\delta+c\overline b,\qquad
v=\overline\beta+\overline a d.$$
Then the relations in \eqref{ext-unitary-conditions} become
$$\Delta=|\lambda|^2,\qquad
\Delta\overline\alpha+\overline\gamma\lambda^2=0,\qquad
\overline u+2\overline\gamma u=-\overline\lambda v.$$
The first two imply $|\alpha|=|\gamma|$, hence $1-4|\alpha|^2=\Delta$,
and
$$\overline\gamma+\frac{\overline\alpha\overline\lambda^2}{\Delta}=0.$$
Combining $\overline u+2\overline\gamma u=-\overline\lambda v$ with its complex conjugate gives
$$\Delta\overline u=-\overline\lambda v+2\overline\gamma\lambda\overline v
=-\overline\lambda(v+2\overline\alpha\overline v),$$
so that
$$\overline u+\frac{\overline\lambda(v+2\overline\alpha\overline v)}{\Delta}=0.$$
On the other hand, the reproducing property and the definition of the adjoint give
\begin{align*}
(T^*K_w)(z)
&=\langle T^*K_w,K_z\rangle
=\langle K_w,TK_z\rangle\\
&=\overline{\langle TK_z,K_w\rangle}
=\overline{(TK_z)(w)},\qquad z,w\in\C.
\end{align*}
Interchanging $z$ and $w$ in \eqref{ext-Gaussian-kernel}
and using the definitions of $\lambda,u,v$, we have
$$
(TK_z)(w)=\overline{f(0)}g(0)
\exp\{\overline b d+\gamma w^2+uw
+\overline\alpha\overline z^2+v\overline z
+\lambda w\overline z\}.
$$
Taking the complex conjugate of this expression, we obtain
\begin{align*}
T^*K_w(z)=f(0)\overline{g(0)}
\exp\{b\overline d+\alpha z^2+\overline v z
+\overline\gamma\overline w^2+\overline u\overline w
+\overline\lambda z\overline w\}.
\end{align*}
Apply \eqref{ext-Gaussian-integral} with
$$\theta=\alpha,\qquad p=\overline v+\overline\lambda\overline w,
\qquad q=v+\lambda\omega.$$
The preceding relations make the coefficients of $\overline w^2$,
$\omega^2$, $\overline w$ and $\omega$ vanish, while the coefficient
of $\overline w\omega$ is $\frac{|\lambda|^2}{\Delta}=1$. Therefore
$$\langle T^*K_w,T^*K_\omega\rangle=C e^{\overline w\omega},
\qquad C=\|T^*K_0\|^2>0.$$
By density of the kernel span, this identity gives $TT^*=CI$.
Since $T^*T=I$, we have $$CT=(TT^*)T=T,$$ so $C=1$.
Thus $TT^*=I$, and $T$ is onto and hence unitary.
This proves $(3)\Rightarrow(2)$. Since $(2)\Rightarrow(1)$ is immediate, we finish the  proof.
\end{proof}

\begin{thm}\label{ext-invertible}
Suppose that $f$ and $g$ are entire functions with $f(0)g(0)\neq0$.
Assume that the natural product $W_{g,\psi}W_{f,\varphi}^*$ has a unique
bounded extension $T$ to $\mathcal F^2$, and that exactly one of
$W_{f,\varphi}$ and $W_{g,\psi}$ is bounded on $\mathcal F^2$.
Then the following assertions are equivalent:
\begin{enumerate}
\item[$(1)$] $T$ is invertible.
\item[$(2)$] $T$ is a nonzero scalar multiple of a unitary operator.
\item[$(3)$] There exist $\alpha,\beta,\gamma,\delta\in\C$ with
$|\alpha|,|\gamma|<\frac{1}{2}$ such that
$$f(z)=f(0)e^{\alpha z^2+\beta z},\qquad
g(z)=g(0)e^{\gamma z^2+\delta z}.$$
With $\Delta=1-4|\gamma|^2$, the identities in
\eqref{ext-unitary-conditions} hold for every $z\in\C$.
\end{enumerate}
\end{thm}
\begin{proof}
By Theorem \ref{ext-affine}, $f,g\in\mathcal F^2$ and there exist $a,b,c,d\in\C$ such that
$$\varphi(z)=az+b,\qquad\psi(z)=cz+d,\qquad0<|ac|\leqslant1.$$
Thus $\varphi^{\prime}(z)=a$, $\psi^{\prime}(z)=c$, $\varphi(0)=b$ and $\psi(0)=d$.
We prove $(1)\Rightarrow(3)\Rightarrow(2)\Rightarrow(1)$.
To prove $(1)\Rightarrow(3)$, suppose that $T$ is invertible.

The kernel formula \eqref{ext-kernel} in Lemma \ref{ext-domain} shows
that $f$ has no zero.
We next show that $g$ has no zeros. Suppose that $g(z_0)=0$
for some $z_0\in\C$. Using
$(T^*K_{z_0})(w)=\overline{(TK_w)(z_0)}$
and the kernel formula \eqref{ext-kernel}, we obtain
\begin{align*}
(T^*K_{z_0})(w)
&=\overline{(TK_w)(z_0)}=f(w)\overline{g(z_0)}
e^{\varphi(w)\overline{\psi(z_0)}}=0,
\qquad w\in\C.
\end{align*}
Thus $T^*K_{z_0}=0$. It follows that $K_{z_0}=0$, since $T^*$ is invertible and hence is injective. This is impossible because
$K_{z_0}(0)=1$. Therefore, both $f$ and $g$ have no zeros.
Thus $f$ and $g$ admit the forms
\begin{align}\label{ext-invertible-weights}
f(z)=f(0)e^{\alpha z^2+\beta z},\qquad
g(z)=g(0)e^{\gamma z^2+\delta z},\qquad
|\alpha|,|\gamma|<\frac{1}{2}.
\end{align}
Let $\Delta=1-4|\gamma|^2>0$. For the weights in
\eqref{ext-invertible-weights}, the inner product identity
\eqref{ext-kernel-inner-product} applies. Setting $w=\omega=0$
in \eqref{ext-kernel-inner-product}, we have
\begin{align}\label{ext-invertible-scale}
\begin{split}
\|TK_0\|^2
={}&\frac{|f(0)g(0)|^2e^{2\mathrm{Re}(\overline{b}d)}}{\sqrt{\Delta}}\exp\left\{\frac{|\delta+c\overline{b}|^2+
2\mathrm{Re}\big[\overline{\gamma}(\delta+c\overline{b})^2\big]}
{\Delta}\right\}>0.
\end{split}
\end{align}
Taking $\omega=w$ in \eqref{ext-kernel-inner-product} and using
$k_w=e^{-\frac{|w|^2}{2}}K_w$, we obtain
\begin{align}\label{ext-invertible-kernel-norm}
\|Tk_w\|^2=\|TK_0\|^2e^{Q(w)},\qquad w\in\C,
\end{align}
where
\begin{align}\label{ext-invertible-quadratic}
\begin{split}
Q(w)={}&\left(\frac{|ac|^2}{\Delta}-1\right)|w|^2
+2\mathrm{Re}\left[
\left(\overline{\alpha}
+\frac{\overline{\gamma}(\overline{a}c)^2}{\Delta}\right)
\overline{w}^2\right]\\
&+2\mathrm{Re}\left[\left(
\overline{\beta}+\overline{a}d
+\frac{\overline{a}c\big[\overline{\delta}+\overline{c}b
+2\overline{\gamma}(\delta+c\overline{b})\big]}{\Delta}
\right)\overline{w}\right].
\end{split}
\end{align}
In particular, $Q$ is a real polynomial in $\mathrm{Re}\,w$
and $\mathrm{Im}\,w$ of degree at most $2$, and $Q(0)=0$.

Since $\|k_w\|=1$, we have
$$1=\|T^{-1}Tk_w\|\leqslant\|T^{-1}\|\|Tk_w\|,$$
while $\|Tk_w\|\leqslant\|T\|$. Consequently,
\begin{align}\label{ext-invertible-kernel-bounds}
\frac{1}{\|T^{-1}\|^2}\leqslant\|Tk_w\|^2\leqslant\|T\|^2,
\qquad w\in\C.
\end{align}
Combining \eqref{ext-invertible-kernel-norm} and
\eqref{ext-invertible-kernel-bounds}, and taking logarithms, we have
$$\log\frac{1}{\|TK_0\|^2\|T^{-1}\|^2}
\leqslant Q(w)\leqslant\log\frac{\|T\|^2}{\|TK_0\|^2},
\qquad w\in\C.$$
Note that a real polynomial bounded above and below on $\R^2$ is constant. Since $Q(0)=0$, it follows that $Q\equiv0$.
Writing $w=x+\mathrm{i}y$ in \eqref{ext-invertible-quadratic} and comparing
coefficients in $Q(x+\mathrm{i}y)=0$, we obtain
\begin{align*}
\frac{|ac|^2}{\Delta}-1=0,\ \ \
\overline{\alpha}
+\frac{\overline{\gamma}(\overline{a}c)^2}{\Delta}=0,
\end{align*}
and
\begin{align*}
\overline{\beta}+\overline{a}d
+\frac{\overline{a}c\big[\overline{\delta}+\overline{c}b
+2\overline{\gamma}(\delta+c\overline{b})\big]}{\Delta}=0.
\end{align*}
Multiplying by $\Delta$ gives precisely
\eqref{ext-unitary-conditions}. Thus $(1)\Rightarrow(3)$.

We next prove $(3)\Rightarrow(2)$. Denote $m=\|TK_0\|>0$ and
$\widetilde g=\frac{g}{m}$. The parameter relations
in \eqref{ext-unitary-conditions} still hold. By
\eqref{ext-invertible-scale}, the left hand side of
\eqref{ext-normalization} for $f,\widetilde g$ is
$$\frac{\|TK_0\|^2}{m^2}=1.$$
Moreover, $W_{\widetilde g,\psi}=m^{-1}W_{g,\psi}$ has the same domain
as $W_{g,\psi}$, and its natural product with $W_{f,\varphi}^*$
has bounded extension $m^{-1}T$. The assumption that exactly one
factor is bounded is also preserved.
Theorem \ref{ext-unitary} therefore shows that $m^{-1}T$ is unitary.
Thus $T$ is a nonzero scalar multiple of a unitary operator.

Finally, suppose that $(2)$ holds. Write $T=\kappa U$ with $\kappa\neq0$ and $U$ unitary. Then $T$ is invertible, with
$$T^{-1}=\frac{1}{\kappa}U^*.$$
This proves $(2)\Rightarrow(1)$ and completes the proof.
\end{proof}

The next example shows that the condition $|ac|=1$ in Theorem \ref{I} is no longer necessary for invertibility when an unbounded factor is allowed.

\begin{exam}\label{ext-squeezing}
Let $0<t<1$ and $s=\sqrt{1-t^2}$ and define
$$f(z)=e^{-\frac{t}{2}z^2},\qquad
g(z)=\sqrt{s}\,e^{\frac{t}{2}z^2},$$
$$\varphi(z)=\sqrt{1-t}\,z,\qquad
\psi(z)=\sqrt{1+t}\,z.$$
Consider $W_{f,\varphi}$ and $W_{g,\psi}$ on their maximal domains.
Since $t<1$, the estimates
$$|f(z)|^2e^{-|z|^2}\leqslant e^{-(1-t)|z|^2}\ \ \ \text{and}\ \ \
|g(z)|^2e^{-|z|^2}\leqslant s e^{-(1-t)|z|^2}$$
imply $f,g\in\mathcal F^2$, and $f(0)g(0)=\sqrt{s}>0$.
Moreover,
\begin{align*}
|f(z)|^2e^{|\varphi(z)|^2-|z|^2}
&=e^{-t\mathrm{Re}(z^2)-t|z|^2}
=e^{-2t[\mathrm{Re}\,(z)]^2}\leqslant1.
\end{align*}
Thus $W_{f,\varphi}$ is bounded by Theorem \ref{L1}.

For every polynomial $P$, we have
$$\|g\cdot(P\circ\psi)\|^2
\leqslant\frac{s}{\pi}\int_{\C}
|P(\sqrt{1+t}\,z)|^2e^{-(1-t)|z|^2}dA(z)<\infty.$$
Hence all polynomials belong to $D(W_{g,\psi})$ and $W_{g,\psi}$
is densely defined. For $x\in\R$,
$$\lim_{x\rightarrow\infty}|g(x)|^2e^{|\psi(x)|^2-|x|^2}
=\lim_{x\rightarrow\infty}s e^{2tx^2}=\infty.$$
Theorem \ref{L1} therefore shows that $W_{g,\psi}$ is unbounded.

Let $T_0=W_{g,\psi}W_{f,\varphi}^*$ on its natural domain.
The reproducing property gives
$$W_{f,\varphi}^*K_w
=\overline{f(w)}K_{\varphi(w)}
=e^{-\frac{t}{2}\overline{w}^2}K_{\sqrt{1-t}\,w}.$$
For each $w\in\C$, formula \eqref{ext-Gaussian-integral} with
$\theta=\frac{t}{2}$, $p=s\overline{w}$ and $q=sw$ gives
\begin{align*}
\|g\cdot(K_{\varphi(w)}\circ\psi)\|^2
&=s\int_{\C}e^{\frac{t}{2}z^2+\frac{t}{2}\overline{z}^2
+s\overline{w}z+sw\overline{z}}d\mu(z)
=e^{|w|^2+t\mathrm{Re}(w^2)}<\infty.
\end{align*}
It follows that $K_{\varphi(w)}\in D(W_{g,\psi})$, and hence
$W_{f,\varphi}^*K_w=\overline{f(w)}K_{\varphi(w)}\in D(W_{g,\psi})$.
Since $W_{f,\varphi}^*$ is bounded, the natural product has domain
$$D(T_0)=D(W_{g,\psi}W_{f,\varphi}^*)
=\{h\in\mathcal F^2:
W_{f,\varphi}^{*}h\in D(W_{g,\psi})\}.$$
Therefore $K_w\in D(T_0)$, and
\begin{align}\label{ext-squeezing-kernel}
T_0K_w(z)=\sqrt{s}\exp\left\{\frac{t}{2}z^2
-\frac{t}{2}\overline{w}^2+s z\overline{w}\right\}.
\end{align}
For $w,\omega\in\C$, applying \eqref{ext-Gaussian-integral} to
\eqref{ext-squeezing-kernel} with
$$\theta=\frac{t}{2},\qquad p=s\overline{w},\qquad q=s\omega,$$
and using $1-t^2=s^2$, we obtain
\begin{align*}
\langle T_0K_w,T_0K_\omega\rangle
&=s e^{-\frac{t}{2}(\overline{w}^2+\omega^2)}
\int_{\C}e^{\frac{t}{2}z^2+\frac{t}{2}\overline{z}^2
+s\overline{w}z+s\omega\overline{z}}d\mu(z)\notag\\
&=\frac{s}{\sqrt{1-t^2}}\exp\left\{
-\frac{t}{2}(\overline{w}^2+\omega^2)
+\frac{s^2\overline{w}\omega+
\frac{ts^2}{2}(\overline{w}^2+\omega^2)}{1-t^2}\right\}\notag\\
&=e^{\overline{w}\omega}=\langle K_w,K_\omega\rangle.
\end{align*}
Thus
$$\|T_0h\|=\|h\|,\qquad h\in E,$$ where $E=\operatorname{span}\{K_w:w\in\C\}.$
Therefore, $T_0|_E$ extends uniquely to a bounded isometry $T$
on $\mathcal F^2$. We first verify that $T$ extends $T_0$ on its
natural domain. For $h\in D(T_0)$, choose $h_n\in E$ such that
$\lim\limits_{n\rightarrow\infty}\|h_n-h\|=0$.
Boundedness of $T$ and $W_{f,\varphi}^*$ implies
\begin{align*}
(Th)(z)
&=\lim_{n\rightarrow\infty}(T_0h_n)(z)\\
&=g(z)\lim_{n\rightarrow\infty}(W_{f,\varphi}^*h_n)(\psi(z))\\
&=g(z)(W_{f,\varphi}^*h)(\psi(z))=(T_0h)(z),\ \ \ z\in \C.
\end{align*}
Since $E\subseteq D(T_0)$ is dense, $T$ is the unique bounded
extension of $T_0$. Lemma \ref{ext-domain} now gives
$$D(T_0)=D(W_{f,\varphi}^*)=\mathcal F^2,$$
so $T_0=T$. Applying $(1)\Rightarrow(2)$ in Theorem \ref{ext-unitary},
we conclude that $T$ is unitary, and therefore invertible, whereas
$$|\varphi^{\prime}(z)\psi^{\prime}(z)|
=\sqrt{1-t}\sqrt{1+t}=s<1,\qquad z\in\C,$$
finishing the proof.
\end{exam}

To discuss in the following  when the operator \(W_{g,\psi}W_{f,\varphi}^{*}\) becomes a Hilbert--Schmidt operator, we first review some preliminaries.

Let $e_n(z)=\frac{z^n}{\sqrt{n!}}$ $(n=0,1,2,\cdots)$ be the standard
orthonormal basis of $\mathcal F^2$. A bounded operator $T$ on
$\mathcal F^2$ is called Hilbert--Schmidt if
$$\sum_{n=0}^{\infty}\|Te_n\|^2<\infty.$$
The class of these operators is denoted by $\mathcal S_2(\mathcal F^2)$,
and its norm is defined by
$$\|T\|_{\mathcal S_2}
=\left(\sum_{n=0}^{\infty}\|Te_n\|^2\right)^{\frac{1}{2}}.$$
This sum is independent of the choice of orthonormal basis.
Every Hilbert--Schmidt operator is compact.

For a nonzero $h\in\mathcal F^2$, the estimate
$|h(z)|\leqslant\|h\|e^{\frac{|z|^2}{2}}$ gives $\rho(h)\leqslant2$.
Using the order $\rho(h)$ and type $\sigma(h)$ defined in
Section \ref{S2}, we write
$$\sigma_2(h)=\limsup_{r\rightarrow\infty}\frac{\log M_h(r)}{r^2}
=\begin{cases}
0,&\rho(h)<2,\vspace{2mm}\\
\sigma(h),&\rho(h)=2.
\end{cases}$$
Thus $0\leqslant\sigma_2(h)\leqslant\frac{1}{2}$.

\begin{prop}\label{ext-HS}
Suppose that $f$ and $g$ are entire functions with $f(0)g(0)\neq0$.
Assume that the natural product $W_{g,\psi}W_{f,\varphi}^*$ has a unique
bounded extension $T$ to $\mathcal F^2$, and that exactly one of
$W_{f,\varphi}$ and $W_{g,\psi}$ is bounded on $\mathcal F^2$.
Then the following assertions hold.\\
$(1)$ $T$ is Hilbert--Schmidt if and only if
\begin{align}\label{ext-HS-integral}
\iint_{\C^2}|f(w)|^2|g(z)|^2
e^{2\mathrm{Re}(\overline{\varphi(w)}\psi(z))}d\mu(z)d\mu(w)<\infty.
\end{align}
If \eqref{ext-HS-integral} holds, the integral in
\eqref{ext-HS-integral} equals $\|T\|_{\mathcal S_2}^2$.\\
$(2)$ If
\begin{align}\label{ext-type-condition}
(1-2\sigma_2(f))(1-2\sigma_2(g))>|\varphi^{\prime}(z)\psi^{\prime}(z)|^2,\qquad z\in\C,
\end{align}
then $T$ is Hilbert--Schmidt and hence compact.
\end{prop}

\begin{proof}
By Theorem \ref{ext-affine}, $f,g\in\mathcal F^2$ and there exist $a,b,c,d\in\C$ such that
$$\varphi(z)=az+b,\qquad\psi(z)=cz+d,\qquad0<|ac|\leqslant1.$$
Consequently, $|\varphi^{\prime}(z)\psi^{\prime}(z)|=|ac|$ for every $z\in\C$.

We first prove $(1)$.
Recall that
$$(T^*K_z)(w)=\langle T^*K_z,K_w\rangle
=\langle K_z,TK_w\rangle=\overline{(TK_w)(z)}.$$
For each $z\in\C$, Parseval's identity therefore gives
\begin{align*}
\sum_{n=0}^{\infty}|(Te_n)(z)|^2
&=\sum_{n=0}^{\infty}|\langle e_n,T^*K_z\rangle|^2
=\|T^*K_z\|^2\\
&=\int_{\C}|(TK_w)(z)|^2d\mu(w).
\end{align*}
Integrating in $z$ and applying Tonelli's theorem,  we obtain
\begin{align}\label{ext-HS-kernel-identity}
\begin{split}
\sum_{n=0}^{\infty}\|Te_n\|^2
&=\iint_{\C^2}|(TK_w)(z)|^2d\mu(z)d\mu(w)\\
&=\int_{\C}\|TK_w\|^2d\mu(w)
=\frac{1}{\pi}\int_{\C}\|Tk_w\|^2dA(w)\\
&=\iint_{\C^2}|f(w)|^2|g(z)|^2
e^{2\mathrm{Re}(\overline{\varphi(w)}\psi(z))}d\mu(z)d\mu(w).
\end{split}
\end{align}
Thus $(1)$ holds.

We now prove $(2)$. Assume that \eqref{ext-type-condition} holds,
and let $\mathcal I$ denote the integral in \eqref{ext-HS-integral}.
By $(1)$, it suffices to prove that $\mathcal I<\infty$.
Choose $\varepsilon>0$ small enough such that
\begin{align}\label{ext-HS-decay-parameters}
p=1-2\sigma_2(f)-2\varepsilon>0,\qquad
q=1-2\sigma_2(g)-2\varepsilon>0,\qquad pq>|ac|^2.
\end{align}
By the definition of $\sigma_2$,
$$\sigma_2(f)=\limsup_{r\rightarrow\infty}\frac{\log M_f(r)}{r^2}.$$
Thus, for any $\varepsilon>0$, there exists $R_f\in \N$ such that
$$ \frac{\log M_f(r)}{r^2}\leqslant \sigma_2(f)+\varepsilon,\ \ \ r\geqslant R_f.$$
It follows that
$$\log M_f(r)\leqslant (\sigma_2(f)+\varepsilon)r^2,$$
which is equivalent to
$$M_f(r)\leqslant e^{(\sigma_2(f)+\varepsilon)r^2}.$$
Thus if $|w|> R_f$, we have
$$
|f(w)|\leqslant M_f(|w|)
\leqslant e^{(\sigma_2(f)+\varepsilon)|w|^2}.$$
Since $f$ is continuous on the closed disk $\{w\in\C:|w|\leqslant R_f\}$,
we can take
$$C_f=\max\Big\{1,\max_{|w|\leqslant R_f}|f(w)|\Big\}<\infty.$$
Note that $\sigma_2(f)+\varepsilon>0$. For $|w|\leqslant R_f$, we have
$$|f(w)|\leqslant C_f
\leqslant C_f e^{(\sigma_2(f)+\varepsilon)|w|^2}.$$
For $|w|>R_f$, the estimate above and $C_f\geqslant1$ give
$$|f(w)|\leqslant e^{(\sigma_2(f)+\varepsilon)|w|^2}
\leqslant C_f e^{(\sigma_2(f)+\varepsilon)|w|^2}.$$
Thus $$|f(w)|\leqslant C_f e^{(\sigma_2(f)+\varepsilon)|w|^2},\ \ \ w\in\C.$$
Applying the same argument to $g$, we obtain a constant $C_g>0$.
Consequently,
\begin{align}\label{ext-HS-growth-bounds}
|f(w)|\leqslant C_f e^{(\sigma_2(f)+\varepsilon)|w|^2},\qquad
|g(w)|\leqslant C_g e^{(\sigma_2(g)+\varepsilon)|w|^2},
\qquad w\in\C.
\end{align}
Since
$$\overline{\varphi(w)}\psi(z)
=\overline{a}c\overline{w}z+\overline{a}d\overline{w}
+c\overline{b}z+\overline{b}d,$$
 \eqref{ext-HS-growth-bounds} yields
\begin{align}\label{ext-HS-Gaussian-bound}
\mathcal I
&\leqslant\frac{C_f^2C_g^2e^{2\mathrm{Re}(\overline{b}d)}}{\pi^2}
\int_{\C}e^{-q|z|^2+2\mathrm{Re}(c\overline{b}z)}\left(\int_{\C}
e^{-p|w|^2+2\mathrm{Re}[\overline{a}(cz+d)\overline{w}]}
dA(w)\right)dA(z).
\end{align}
Observe that
$$-p|w|^2+2\mathrm{Re}[\overline{a}(cz+d)\overline{w}]
=-p\left|w-\frac{\overline{a}(cz+d)}{p}\right|^2
+\frac{|a|^2|cz+d|^2}{p}.$$
Thus we have
$$\int_{\C}e^{-p|w|^2+2\mathrm{Re}[\overline{a}(cz+d)\overline{w}]}
dA(w)=\frac{\pi}{p}e^{\frac{|a|^2|cz+d|^2}{p}}.$$
Let
$$r=q-\frac{|ac|^2}{p}>0,\qquad
\ell=c\overline{b}+\frac{|a|^2c\overline{d}}{p},$$
where $r>0$ follows from \eqref{ext-HS-decay-parameters}.
Then
\begin{align*}
-q|z|^2+2\mathrm{Re}(c\overline{b}z)
+\frac{|a|^2|cz+d|^2}{p}
&=-r|z|^2+2\mathrm{Re}(\ell z)+\frac{|ad|^2}{p}\\
&=-r\left|z-\frac{\overline{\ell}}{r}\right|^2
+\frac{|\ell|^2}{r}+\frac{|ad|^2}{p}.
\end{align*}
Substituting the above equality  into \eqref{ext-HS-Gaussian-bound}
and using
$$
\int_{\mathbb C}e^{-r\left|z-\frac{\overline\ell}{r}\right|^2}\,dA(z)
=\int_{\mathbb C}e^{-r|u|^2}\,dA(u)
=\frac{\pi}{r},
$$
we obtain
$$\mathcal I\leqslant\frac{C_f^2C_g^2}{pr}
\exp\left\{2\mathrm{Re}(\overline{b}d)
+\frac{|ad|^2}{p}+\frac{|\ell|^2}{r}\right\}<\infty.$$
Hence \eqref{ext-HS-integral} holds. Applying $(1)$, we conclude
that $T$ is Hilbert--Schmidt and hence compact.
\end{proof}

The next example shows that the assumption in Proposition \ref{T2} that both factors are bounded
cannot be replaced by boundedness of just one factor. In this example,
$$|\varphi^{\prime}(z)\psi^{\prime}(z)|<1,\qquad
\varphi(0)\psi(0)\neq0,\qquad \rho(f)+\rho(g)<3,$$
but $W_{g,\psi}W_{f,\varphi}^*$ is bounded and not compact.

\begin{exam}\label{ext-noncompact}
Let
$$f(z)=1,\qquad\varphi(z)=\frac{z}{2}+1,\qquad
g(z)=e^{-1}\exp\left\{\frac{3z^2}{8}-\frac{5z}{4}\right\},\qquad
\psi(z)=z+1,$$
and consider $W_{f,\varphi}$ and $W_{g,\psi}$ on their maximal domains. By  Theorem \ref{L1}, $W_{f,\varphi}$ is bounded.
We next show that $W_{g,\psi}$ is densely defined but unbounded.
For $z=x+\mathrm{i}y$, we have
\begin{align*}
|g(z)|^2e^{-|z|^2}
&=\exp\left\{-2-\frac{x^2}{4}-\frac{7y^2}{4}-\frac{5x}{2}\right\}\\
&=e^{\frac{17}{4}}
\exp\left\{-\frac{(x+5)^2}{4}-\frac{7y^2}{4}\right\}.
\end{align*}
Let $P$ be a polynomial of degree at most $N$. Choose $C>0$ such that
$$|P(u-4)|^2\leqslant C(1+|u|^{2N}),\qquad u\in\C.$$
We make the change of variables $u=z+5$. Since $z=x+\mathrm{i}y$, we have
$$\mathrm{Re}\,u=x+5,\qquad
\mathrm{Im}\,u=y,\qquad dA(u)=dA(z),$$
and
$$P(\psi(z))=P(z+1)=P(u-4).$$
Substituting these identities into the expansion of
$\|g\cdot(P\circ\psi)\|^2$ and then using polar coordinates, we obtain
\begin{align*}
\|g\cdot(P\circ\psi)\|^2&=\frac{1}{\pi}\int_\C\big|g(z)P(\psi(z))\big|^2e^{-|z|^2}dA(z)\\
&=\frac{e^{\frac{17}{4}}}{\pi}\int_{\C}|P(u-4)|^2
e^{-\frac{(\mathrm{Re}\,u)^2}{4}-\frac{7(\mathrm{Im}\,u)^2}{4}}dA(u)\\
&\leqslant\frac{Ce^{\frac{17}{4}}}{\pi}
\int_{\C}(1+|u|^{2N})e^{-\frac{|u|^2}{4}}dA(u)\\
&=2Ce^{\frac{17}{4}}\int_0^{\infty}(1+R^{2N})e^{-\frac{R^2}{4}}R\,dR\\
&=Ce^{\frac{17}{4}}(4+4^{N+1}N!)<\infty.
\end{align*}
Taking $P=1$ gives $g\in\mathcal F^2$. Moreover, every polynomial belongs to
$D(W_{g,\psi})$, so $W_{g,\psi}$ is densely defined.
On the real axis,
$$|g(x)|^2e^{|\psi(x)|^2-|x|^2}
=e^{\frac{3x^2}{4}-\frac{x}{2}-1},\qquad x\in\R,$$
and hence
$$\lim_{x\rightarrow\infty}|g(x)|^2e^{|\psi(x)|^2-|x|^2}
=\lim_{x\rightarrow\infty}e^{\frac{3x^2}{4}-\frac{x}{2}-1}=\infty.$$
Therefore,  Theorem \ref{L1} implies that $W_{g,\psi}$ is unbounded.

To compute the product, first observe that
$$(W_{f,\varphi}K_z)(u)
=e^{\overline{z}(\frac{u}{2}+1)}
=e^{\overline{z}}K_{\frac{z}{2}}(u).$$
For every $h\in\mathcal F^2$, the reproducing property gives
\begin{align*}
(W_{f,\varphi}^*h)(z)
&=\langle W_{f,\varphi}^*h,K_z\rangle
=\langle h,W_{f,\varphi}K_z\rangle\\
&=\langle h,e^{\overline{z}}K_{\frac{z}{2}}\rangle
=e^zh\left(\frac{z}{2}\right).
\end{align*}
Let $T_0=W_{g,\psi}W_{f,\varphi}^*$ on its natural domain.
For every $h\in\operatorname{D}(T_0)$, we have the pointwise identity
\begin{align}\label{ext-noncompact-product-formula}
g(z)(W_{f,\varphi}^*h)(\psi(z))
&=e^{-1}\exp\left\{\frac{3z^2}{8}-\frac{5z}{4}\right\}
e^{z+1}h\left(\frac{z+1}{2}\right)\notag\\
&=\exp\left\{\frac{3z^2}{8}-\frac{z}{4}\right\}
h\left(\frac{z+1}{2}\right).
\end{align}
We now verify that the right hand side defines a bounded weighted
composition operator on $\mathcal F^2$. For $z=x+\mathrm{i}y$,
\begin{align}\label{ext-noncompact-bound}
\left|\exp\left\{\frac{3z^2}{8}-\frac{z}{4}\right\}\right|^2
e^{\left|\frac{z+1}{2}\right|^2-|z|^2}
&=\exp\left\{\frac{3}{4}\mathrm{Re}(z^2)
-\frac{3}{4}|z|^2+\frac{1}{4}\right\}\notag\\
&=e^{\frac{1}{4}-\frac{3y^2}{2}}\leqslant e^{\frac{1}{4}}.
\end{align}
In particular, the weight in \eqref{ext-noncompact-product-formula}
belongs to $\mathcal F^2$, since \eqref{ext-noncompact-bound} gives
$$\int_{\C}\left|\exp\left\{\frac{3z^2}{8}-\frac{z}{4}\right\}\right|^2d\mu(z)
\leqslant\frac{e^{\frac{1}{4}}}{\pi}
\int_{\C}e^{-\left|\frac{z+1}{2}\right|^2}dA(z)
=4e^{\frac{1}{4}}<\infty.$$
By Theorem \ref{L1}, the weighted composition operator
$$T=W_{\exp\{\frac{3z^2}{8}-\frac{z}{4}\},\,\frac{z+1}{2}}$$
is bounded on $\mathcal F^2$.
Equation \eqref{ext-noncompact-product-formula} shows that
$T$ agrees with $T_0$ on $D(T_0)$.
To verify that this domain is dense, apply the formula for
$W_{f,\varphi}^*h$ above to $h=K_w$. It gives
$$g\cdot\bigl((W_{f,\varphi}^*K_w)\circ\psi\bigr)
=TK_w\in\mathcal F^2,\qquad w\in\C.$$
Thus $K_w\in D(T_0)$ for every $w\in\C$, and $D(T_0)$ is dense.
Consequently, $T$ is the unique bounded extension of $T_0$.
Since $W_{f,\varphi}^*$ is bounded, Lemma \ref{ext-domain} gives
$$D(T_0)=D(W_{f,\varphi}^*)=\mathcal F^2.$$
Hence $T_0=T$, and
$$(Th)(z)=\exp\left\{\frac{3z^2}{8}-\frac{z}{4}\right\}
h\left(\frac{z+1}{2}\right),\qquad h\in\mathcal F^2.$$

Taking $z=x\in\R$ in \eqref{ext-noncompact-bound}, we obtain
$$\lim_{x\rightarrow\infty}
\left|\exp\left\{\frac{3x^2}{8}-\frac{x}{4}\right\}\right|^2
e^{\left|\frac{x+1}{2}\right|^2-x^2}=e^{\frac{1}{4}}>0.$$
Hence $T$ is not compact by part $(2)$ of Theorem \ref{L1}.

Finally, $\rho(f)=0$, $\rho(g)=2$, and
$$|\varphi^{\prime}(z)\psi^{\prime}(z)|=\frac{1}{2}<1,\qquad
\varphi(0)\psi(0)=1\neq0,\qquad \rho(f)+\rho(g)=2<3.$$
Thus the three conditions in (2) of Proposition \ref{T2} hold,
while the product is not compact. The assumption that both factors
are bounded therefore cannot be replaced by boundedness of just one factor.
\end{exam}

Proposition \ref{ext-HS} gives a sufficient condition for compactness for general weights.
For quadratic exponential weights, the next proposition gives a necessary and sufficient
condition. It also shows that compactness and the Hilbert--Schmidt property are equivalent
for this class of products.

\begin{prop}\label{ext-Gaussian-compact}
Let
$$f(z)=f(0)e^{\alpha z^2+\beta z},\qquad
g(z)=g(0)e^{\gamma z^2+\delta z},$$
where $f(0)g(0)\neq0$, $\alpha,\beta,\gamma,\delta\in\C$ and
$|\alpha|,|\gamma|<\frac{1}{2}$.
Assume that $W_{g,\psi}W_{f,\varphi}^*$ has a unique bounded extension $T$
to $\mathcal F^2$. Assume also that exactly one of $W_{f,\varphi}$ and
$W_{g,\psi}$ is bounded on $\mathcal F^2$.
Put $\Delta=1-4|\gamma|^2>0$. Then the following assertions are equivalent:\\
$(1)$ $T$ is compact.\\
$(2)$ $T$ is Hilbert--Schmidt.\\
$(3)$ For every $z\in\C$,
$$\Delta-|\varphi^{\prime}(z)\psi^{\prime}(z)|^2
>2\left|\Delta\overline{\alpha}
+\overline{\gamma}\big(\overline{\varphi^{\prime}(z)}\psi^{\prime}(z)\big)^2\right|.$$
\end{prop}
\begin{proof}
By Theorem \ref{ext-affine}, there exist $a,b,c,d\in\C$ such that
$$\varphi(z)=az+b,\qquad\psi(z)=cz+d,\qquad0<|ac|\leqslant1.$$
We prove $(1)\Rightarrow(3)\Rightarrow(2)\Rightarrow(1)$.
For the weights in the statement, the norm identity
\eqref{ext-invertible-kernel-norm} holds with $Q$ given by
\eqref{ext-invertible-quadratic}.
To simplify the coefficients in $Q$, let
\begin{align*}
r&=\Delta-|ac|^2,\qquad
s=\Delta\overline{\alpha}+\overline{\gamma}(\overline{a}c)^2,\\
\ell&=\overline{\beta}+\overline{a}d
+\frac{\overline{a}c\big[\overline{\delta}+\overline{c}b
+2\overline{\gamma}(\delta+c\overline{b})\big]}{\Delta}.
\end{align*}

Then \eqref{ext-invertible-kernel-norm} becomes
\begin{align}\label{ext-Gaussian-compact-norm}
\|Tk_w\|^2=\|TK_0\|^2
\exp\left\{-\frac{r}{\Delta}|w|^2
+\frac{2}{\Delta}\mathrm{Re}(s\overline{w}^2)
+2\mathrm{Re}(\ell\overline{w})\right\},\qquad w\in\C,
\end{align}
where $\|TK_0\|^2>0$ by \eqref{ext-invertible-scale}.
Condition $(3)$ is precisely $r>2|s|$.

We first prove $(1)\Rightarrow(3)$. Suppose that $T$ is compact.
Since the normalized kernels converge weakly to zero as $|w|\rightarrow\infty$,
compactness gives $\lim\limits_{|w|\rightarrow\infty}\|Tk_w\|=0$.
To prove that $r>2|s|$,
we choose a unit direction $\eta$ with $s\overline\eta^2=|s|$.
If $s=0$, take $\eta=1$. If $s\neq0$, write $s=|s|e^{\mathrm{i}\theta}$
and take $\eta=e^{\frac{\mathrm{i}\theta}{2}}$. In either case,
$$|\eta|=1,\qquad s\overline{\eta}^{\,2}=|s|.$$
Substituting $w=t\eta$ and $w=-t\eta$ into
\eqref{ext-Gaussian-compact-norm} gives, for $t>0$,
\begin{align*}
\|Tk_{t\eta}\|^2
&=\|TK_0\|^2
\exp\left\{-\frac{r-2|s|}{\Delta}t^2
+2t\mathrm{Re}(\ell\overline{\eta})\right\},\\
\|Tk_{-t\eta}\|^2
&=\|TK_0\|^2
\exp\left\{-\frac{r-2|s|}{\Delta}t^2
-2t\mathrm{Re}(\ell\overline{\eta})\right\}.
\end{align*}
It follows that
\begin{align}\label{ext-Gaussian-compact-opposite-kernels}
\|Tk_{t\eta}\|^2\|Tk_{-t\eta}\|^2
=\|TK_0\|^4\exp\left\{-\frac{2(r-2|s|)}{\Delta}t^2\right\}.
\end{align}
Since $T$ is compact,
$$\lim_{t\rightarrow\infty}\|Tk_{t\eta}\|^2\|Tk_{-t\eta}\|^2=0.$$
If $r\leqslant2|s|$, \eqref{ext-Gaussian-compact-opposite-kernels} gives
$$\|Tk_{t\eta}\|^2\|Tk_{-t\eta}\|^2
\geqslant\|TK_0\|^4>0,\qquad t>0,$$
contradicting this limit.
Hence $r>2|s|$, which proves $(3)$.

We next prove $(3)\Rightarrow(2)$. Suppose that $r>2|s|$ and let
$$\varepsilon=\frac{r-2|s|}{\Delta}>0.$$
Using $\mathrm{Re}(s\overline{w}^2)\leqslant|s||w|^2$
and
$$2|\ell||w|\leqslant\frac{\varepsilon}{2}|w|^2
+\frac{2|\ell|^2}{\varepsilon},$$
we obtain from \eqref{ext-Gaussian-compact-norm} that
$$\|Tk_w\|^2
\leqslant\|TK_0\|^2e^{-\varepsilon|w|^2+2|\ell||w|}
\leqslant\|TK_0\|^2e^{\frac{2|\ell|^2}{\varepsilon}}
e^{-\frac{\varepsilon}{2}|w|^2}.$$
Therefore \eqref{ext-HS-kernel-identity} gives
\begin{align*}
\sum_{n=0}^{\infty}\|Te_n\|^2
&=\frac{1}{\pi}\int_{\C}\|Tk_w\|^2dA(w)\notag\\
&\leqslant\frac{\|TK_0\|^2}{\pi}
e^{\frac{2|\ell|^2}{\varepsilon}}
\int_{\C}e^{-\frac{\varepsilon}{2}|w|^2}dA(w)\notag\\
&=\frac{2\|TK_0\|^2}{\varepsilon}
e^{\frac{2|\ell|^2}{\varepsilon}}<\infty.
\end{align*}
It follows that $T$ is Hilbert--Schmidt. Thus $(3)\Rightarrow(2)$.
Since every Hilbert--Schmidt operator is compact,
$(2)\Rightarrow(1)$ follows immediately.
This proves the equivalence of $(1)$, $(2)$ and $(3)$.

\end{proof}

The case in which both weights have no zeros is covered by
Proposition \ref{ext-Gaussian-compact}. The next proposition shows that,
when both weights have only finitely many zeros and at least one has a zero,
boundedness of the product extension already implies the Hilbert--Schmidt property.

\begin{prop}\label{ext-finite-zeros}
Suppose that $f$ and $g$ are entire functions with $f(0)g(0)\neq0$,
and that each of them has only finitely many zeros.
Assume that $W_{g,\psi}W_{f,\varphi}^*$ has a unique bounded extension
$T$ to $\mathcal F^2$, and that exactly one of $W_{f,\varphi}$ and
$W_{g,\psi}$ is bounded on $\mathcal F^2$.
If at least one of $f$ and $g$ has a zero, then $T$ is Hilbert--Schmidt
and hence compact.
\end{prop}
\begin{proof}
Suppose that at least one of $f$ and $g$ has a zero.
By Theorem \ref{ext-affine}, $f,g\in\mathcal F^2$ and there exist $a,b,c,d\in\C$ such that
$$\varphi(z)=az+b,\qquad\psi(z)=cz+d,\ \ \ 0<|ac|\leqslant1.$$
Since $f,g\in\mathcal F^2$ have order at most $2$ and only finitely
many zeros, Theorem \ref{hadamard-factorization} gives
$$f(z)=P(z)e^{\alpha z^2+\beta z},\qquad
g(z)=Q(z)e^{\gamma z^2+\delta z},$$
where $P$ and $Q$ are nonzero polynomials with the same zeros as $f$ and $g$,
respectively, counted with multiplicities.
Next we  verify that $|\alpha|<\frac{1}{2}$. Since $P$ is a nonzero polynomial, there exist $R>0$ and $C_P>0$,
both depending only on $P$, such that
$$|P(z)|^{-1}\leqslant C_P$$ for $|z|>R$.
Since $f(z)=P(z)e^{\alpha z^2+\beta z}$, we have
$$
\int_{\mathbb C}|e^{\alpha z^2+\beta z}|^2d\mu(z)
\leqslant
\max_{|z|\leqslant R}|e^{\alpha z^2+\beta z}|^2
+C_P^2\|f\|^2<\infty.
$$
Thus $e^{\alpha z^2+\beta z}\in\mathcal F^2$, and
Lemma \ref{ext-exponential-fock} gives $|\alpha|<\frac{1}{2}$.
Similarly, $|\gamma|<\frac{1}{2}$.
Let $m=\deg P$ and $n=\deg Q$.
Since $f(0)=P(0)$ and $g(0)=Q(0)$, we have $P(0)Q(0)\neq0$.
The assumption that at least one of $f,g$ has a zero implies
that at least one of $P,Q$ is nonconstant, so $m+n>0$.

We shall prove that the integral in \eqref{ext-HS-integral} is finite.
Combining  \eqref{ext-kernel} with  Lemma \ref{ext-domain} gives
\begin{align}\label{ext-finite-zeros-kernel}
\begin{split}
\langle Tk_w,k_z\rangle=\overline{P(w)}Q(z)\exp\bigg\{&\overline{b}d+\gamma z^2+\overline{\alpha}\overline{w}^2
+\overline{a} c z\overline{w}+(\delta+c\overline{b})z\\
&+(\overline{\beta}+\overline{a} d)\overline{w}-\frac{|z|^2+|w|^2}{2}\bigg\},\ \ w, z\in \C.
\end{split}
\end{align}

Let
$$R_1(w,z)=\mathrm{Re}\big(\gamma z^2+\overline{\alpha}\overline{w}^2+\overline{a} c z\overline{w}\big)
-\frac{|z|^2+|w|^2}{2}.$$
Fix $(w_0,z_0)\in\C^2$. For $t>0$, \eqref{ext-finite-zeros-kernel} gives
$$|\langle Tk_{\pm tw_0},k_{\pm tz_0}\rangle|
=|P(\pm tw_0)Q(\pm tz_0)|
e^{\mathrm{Re}(\overline b d)+t^2R_1(w_0,z_0)}\times
e^{\pm t\mathrm{Re}[(\delta+c\overline b)z_0+
(\overline\beta+\overline a d)\overline{w_0}]}.
$$
Multiplying the two equalities above and using $\|k_w\|=1$,
for $t>0$, we have
\begin{align}\label{ext-finite-zeros-opposite-kernels}
\begin{split}
&|P(tw_0)P(-tw_0)Q(tz_0)Q(-tz_0)|
e^{2\mathrm{Re}(\overline{b}d)+2t^2R_1(w_0,z_0)}\\
&\qquad=|\langle Tk_{tw_0},k_{tz_0}\rangle|
|\langle Tk_{-tw_0},k_{-tz_0}\rangle|
\leqslant\|T\|^2.
\end{split}
\end{align}
 Note that $P(0)^2Q(0)^2\neq0$. The polynomial $P(tw_0)P(-tw_0)Q(tz_0)Q(-tz_0)$ is not identically
zero. Hence
$$\lim_{t\rightarrow\infty}
\frac{\log|P(tw_0)P(-tw_0)Q(tz_0)Q(-tz_0)|}{t^2}=0.$$
Thus, taking logarithms in \eqref{ext-finite-zeros-opposite-kernels},
dividing by $t^2$ and letting $t\rightarrow\infty$, we obtain
$$R_1(w_0,z_0)\leqslant0.$$

We next show that the inequality is strict whenever $(w_0,z_0)\neq(0,0)$.
If not, suppose that $R_1(w_0,z_0)=0$ for such a pair.
The estimates
\begin{align*}
R_1(0,z)&\leqslant-\left(\frac{1}{2}-|\gamma|\right)|z|^2<0,
\qquad z\neq0
\end{align*}
and
\begin{align*}
R_1(w,0)&\leqslant-\left(\frac{1}{2}-|\alpha|\right)|w|^2<0,
\qquad w\neq0,
\end{align*}
show that $w_0z_0\neq0$.
Hence both $w_0$ and $z_0$ are nonzero. Therefore the polynomial in
\eqref{ext-finite-zeros-opposite-kernels} has degree $2(m+n)>0$, and
$$\lim_{t\rightarrow\infty}|P(tw_0)P(-tw_0)Q(tz_0)Q(-tz_0)|=\infty.$$
This contradicts \eqref{ext-finite-zeros-opposite-kernels}. Thus $R_1(w,z)<0$ for every nonzero $(w,z)$.
Then by continuity on the unit sphere in $\C^2$, we have
$$\varepsilon:=-\max_{|w|^2+|z|^2=1}R_1(w,z)>0.$$
Since $R_1(tw,tz)=t^2R_1(w,z)$ for real $t$, it follows that
\begin{align}\label{ext-finite-zeros-decay}
R_1(w,z)\leqslant-\varepsilon(|w|^2+|z|^2),\qquad w,z\in\C.
\end{align}

Choose $C_0>0$ such that
$$|P(w)|^2|Q(z)|^2\leqslant C_0(1+|w|^{2m})(1+|z|^{2n}),
\qquad w,z\in\C.$$
The linear terms satisfy
\begin{align*}
\begin{cases}
2\mathrm{Re}\big[(\delta+c\overline{b})z\big]
\leqslant\varepsilon|z|^2+\frac{|\delta+c\overline{b}|^2}{\varepsilon},\vspace{2mm}\\
2\mathrm{Re}\big[(\overline{\beta}+\overline{a}d)\overline{w}\big]
\leqslant\varepsilon|w|^2+\frac{|\overline{\beta}+\overline{a}d|^2}{\varepsilon}.
\end{cases}
\end{align*}
Let $\mathcal I$ denote the integral in \eqref{ext-HS-integral}, and put
$$C=C_0\exp\left\{2\mathrm{Re}(\overline{b}d)
+\frac{|\delta+c\overline{b}|^2+|\overline{\beta}+\overline{a}d|^2}{\varepsilon}\right\}.$$
Using \eqref{ext-finite-zeros-kernel}, \eqref{ext-finite-zeros-decay}
and the preceding estimates, we obtain
\begin{align*}
\mathcal I
&=\frac{1}{\pi^2}\iint_{\C^2}|\langle Tk_w,k_z\rangle|^2dA(z)dA(w)\\
&\leqslant\frac{C}{\pi^2}
\left(\int_{\C}(1+|w|^{2m})e^{-\varepsilon|w|^2}dA(w)\right)
\left(\int_{\C}(1+|z|^{2n})e^{-\varepsilon|z|^2}dA(z)\right).
\end{align*}
For every integer $j\geqslant0$, polar coordinates and the substitution
$u=\varepsilon R^2$ give
$$\int_{\C}|z|^{2j}e^{-\varepsilon|z|^2}dA(z)
=2\pi\int_0^\infty R^{2j+1}e^{-\varepsilon R^2}dR
=\frac{\pi}{\varepsilon^{j+1}}\int_0^\infty u^je^{-u}du
=\frac{\pi j!}{\varepsilon^{j+1}}.$$
Consequently,
$$\mathcal I\leqslant C
\left(\frac{1}{\varepsilon}+\frac{m!}{\varepsilon^{m+1}}\right)
\left(\frac{1}{\varepsilon}+\frac{n!}{\varepsilon^{n+1}}\right)<\infty.$$
Proposition \ref{ext-HS}$(1)$ now implies that $T$ is Hilbert--Schmidt
and hence compact.
\end{proof}

\begin{cor}\label{ext-unweighted}
Assume that $C_\psi C_\varphi^*$ has a unique bounded extension $T$
to $\mathcal F^2$, and that exactly one of $C_\varphi$ and $C_\psi$
is bounded on $\mathcal F^2$.
Then the following assertions hold:\\
$(1)$ $T$ is compact if and only if
 $|\varphi^{\prime}(z)\psi^{\prime}(z)|<1$ for every $z\in\C$.\\
$(2)$ $T$ is invertible if and only if
$|\varphi^{\prime}(z)\psi^{\prime}(z)|=1$ for every $z\in\C$.\\
$(3)$ $T$ is isometric if and only if it is unitary.
These properties hold if and only if $\varphi(0)=\psi(0)=0$ and
$|\varphi^{\prime}(z)\psi^{\prime}(z)|=1$ for every $z\in\C$.
\end{cor}
\begin{proof}
By Theorem \ref{ext-affine}, there exist $a,b,c,d\in\C$ such that
$$\varphi(z)=az+b,\qquad\psi(z)=cz+d,\qquad0<|ac|\leqslant1.$$
Thus $\varphi^{\prime}(z)\psi^{\prime}(z)=ac$ for every $z\in\C$,
and $\varphi(0)=b$, $\psi(0)=d$.
By Lemma \ref{ext-domain}, we have
$$(TK_w)(z)=e^{\overline{aw+b}(cz+d)},\qquad z,w\in\C.$$
Since $(T^*K_z)(w)=\overline{(TK_w)(z)}$, it follows that
$$T^*K_z=e^{b\overline{cz+d}}K_{\overline{a}(cz+d)}.$$
Therefore, the reproducing property gives, for every $h\in\mathcal F^2$,
\begin{align}\label{ext-unweighted-forward}
(Th)(z)=\langle h,T^*K_z\rangle
=e^{\overline{b}(cz+d)}h(\overline{a}(cz+d)),\qquad z\in\C.
\end{align}

We  prove $(1)-(3)$ in sequence.
For the constant weights $f=g=1$, we have
$\alpha=\beta=\gamma=\delta=0$ and $\Delta=1$.
Condition $(3)$ of Proposition \ref{ext-Gaussian-compact} therefore
reduces to $1-|ac|^2>0$. Thus $T$ is compact if and only if $|ac|<1$. Thus conclusion $(1)$ holds.

The symbol in \eqref{ext-unweighted-forward} is
$\overline{a}cz+\overline{a}d$. Since $T$ is bounded,
Theorem \ref{L7} gives that $T$ is invertible if and only if
$|\overline{a}c|=1$, which proves $(2)$.

For $(3)$, Theorem \ref{ext-unitary}, with the constant weights $f=g=1$,
states that $T$ is isometric if and only if it is unitary, and that
these properties hold precisely when
$$|ac|=1,\qquad d+|c|^2b=0,\qquad
e^{2\mathrm{Re}(\overline{b}d)+|c|^2|b|^2}=1.$$
Substituting $d=-|c|^2b$ into the last equality gives
$e^{-|c|^2|b|^2}=1$. Since $c\neq0$, this is equivalent to $b=d=0$.
This proves $(3)$ and completes the proof.
\end{proof}

In Theorem \ref{U}, the products $W_{g,\psi}W_{f,\varphi}^*$ and
$W_{g,\psi}^*W_{f,\varphi}$ are unitary simultaneously when both factors
are bounded. The following example shows that this equivalence can fail
if only one factor is bounded: $W_{g,\psi}C_\varphi^*$ is unitary,
whereas $W_{g,\psi}^*C_\varphi$ has no bounded extension.

\begin{exam}\label{ext-reverse-unbounded}
Let
$$\varphi(z)=\frac{z}{2},\qquad
g(z)=e^{-\frac{1}{2}}e^{-z},\qquad\psi(z)=2z+2,$$
and consider $C_\varphi$ and $W_{g,\psi}$ on their maximal domains.
The operator $C_{\varphi}$ is bounded, whereas $W_{g,\psi}$ is unbounded and its domain contains all polynomials. Direct calculation gives
$$W_{g,\psi}C_{\varphi}^*h(z)=e^{-\frac{1}{2}}e^{-z}h(z+1),\qquad h\in\mathcal F^2.$$
 By Theorem \ref{L6}, $W_{g,\psi}C_{\varphi}^*$ is unitary; in particular, the natural product is defined on all of $\mathcal F^2$. We next show that $W_{g,\psi}^*C_\varphi$ has no bounded extension.
For every $w\in\mathbb C$, we have
\[
(C_\varphi K_w)(z)
=e^{\frac{\overline{w}z}{2}}
=K_{\frac{w}{2}}(z).
\]
By Lemma \ref{ext-adjoint-kernel},
$K_{\frac{w}{2}}\in D(W_{g,\psi}^*)$.
Hence $K_w\in D(W_{g,\psi}^*C_\varphi)$, and
\begin{align*}
W_{g,\psi}^*C_\varphi K_w
&=W_{g,\psi}^*K_{\frac{w}{2}}=\overline{g(\frac{w}{2})}K_{\psi(\frac{w}{2})}
=e^{-\frac12-\frac{\overline{w}}2}K_{w+2}.
\end{align*}
For $x\in\mathbb R$, the normalized kernel
$k_x=e^{-\frac{x^2}{2}}K_x$ satisfies $\|k_x\|=1$.
Using $\|K_{x+2}\|^2=e^{(x+2)^2}$, we obtain
\begin{align*}
\|W_{g,\psi}^*C_\varphi k_x\|^2
=e^{-x^2-1-x}\|K_{x+2}\|^2=e^{-x^2-1-x+(x+2)^2}
=e^{3x+3}.
\end{align*}
Consequently,
$$
\lim_{x\rightarrow\infty}\|W_{g,\psi}^*C_\varphi k_x\|^2=\infty.
$$
Thus $W_{g,\psi}^*C_\varphi$ has no bounded extension,
although $W_{g,\psi}C_\varphi^*$ is unitary.
\end{exam}\vspace{3mm}

\subsection*{Acknowledgment}
This work was supported by the National Natural Science Foundation of China
(12371125)  and  Chongqing Natural Science Foundation (CSTB2024NSCQ-MSX0177).


\begin{thebibliography}{99}

\bibitem{TC}
T. Carroll, C. Gilmore, Weighted composition operators on Fock spaces and their dynamics, \emph{J. Math. Anal. Appl.}, 2021, 502: Paper No. 125234.

\bibitem{JC}
J. H. Clifford, The product of a composition operator with the adjoint of a composition operator, Ph.D. Thesis, Michigan State University, 1998.

\bibitem{CD1}
J. H. Clifford, D. Zheng, Composition operators on the Hardy space, \emph{Indiana Univ. Math. J.}, 1999, 48: 1585--1616.

\bibitem{CD}
J. H. Clifford, D. Zheng, Composition operators on Bergman spaces, \emph{Chinese Ann. Math. Ser. B}, 2003, 24: 433--448.

\bibitem{CL}
J. H. Clifford, T. Le, A. Wiggins, Invertible composition operators: the product of a composition operator with the adjoint of a composition operator, \emph{Complex Anal. Oper. Theory}, 2014, 8: 1699--1706.

\bibitem{CM}
J. H. Clifford, M. Dabkowski, A. Wiggins, Y. Zeytuncu, The operator $ C_ {\phi}^* C_ {\psi} $ when $\phi $ is a Blaschke product, \emph{Adv. Oper. Theory}, 2026, 11: Paper No. 3.

\bibitem{DZ}
X. Dong, K. Zhu, Canonical integral operators on the Fock space, \emph{Math. Z.}, 2024, 306: Paper No. 64.

\bibitem{PVH}
P. V. Hai, Unbounded weighted composition operators on Fock space,  \emph{Potential Anal.}, 2020, 53: 1--21. 

\bibitem{HR-MittagLeffler}
P. V. Hai, J. A. Rosenfeld, Weighted composition operators on the Mittag-Leffler spaces of entire functions, \emph{Complex Anal. Oper. Theory}, 2021, 15: Paper No. 17.

\bibitem{HT-multivariable}
P. V. Hai, P. T. Tien, Complex symmetry in the Fock space of several variables, \emph{Linear Algebra Appl.}, 2024, 695: 49--74.

\bibitem{SH}
M. Haji Shaabani, Fredholmness of multiplication of a weighted composition operator with its adjoint on $H^2$ and $A_\alpha^2$, \emph{J. Inequal. Appl.}, 2018, 2018: 23.

\bibitem{SH2}
M. Haji Shaabani, Closed range weighted composition operators on $H^2$ and $A_\alpha^2$, \emph{Iran. J. Sci. Technol. Trans. A Sci.}, 2019, 43: 1167--1170.

\bibitem{HW}
K. Han, M. Wang, Weighted composition operators on the Fock space, \emph{Sci. China Math.}, 2022, 65: 111--126.

\bibitem{LT}
T. Le, Normal and isometric weighted composition operators on the Fock space, \emph{Bull. Lond. Math. Soc.}, 2014, 46: 847--856.

\bibitem{LZ}
Q. Leng, X. Zhao, Schatten class membership of products of composition operators and their adjoints, arXiv: 2610.15472.




\bibitem{MG-generalized-Fock}
A. Malhotra, A. Gupta, Complex symmetry of generalized weighted composition operators on Fock space, \emph{J. Math. Anal. Appl.}, 2021, 495:  Paper No. 124740.

\bibitem{MT1}
T. Mengestie, Schatten class weighted composition operators on weighted Fock spaces, \emph{Arch. Math. (Basel)}, 2013, 101: 349--360.

\bibitem{MT2}
T. Mengestie, On trace ideal weighted composition operators on weighted Fock spaces, \emph{Arch. Math. (Basel)}, 2015, 105: 453--459.

\bibitem{MT3}
T. Mengestie, Closed range weighted composition operators and dynamical sampling, \emph{J. Math. Anal. Appl.}, 2022, 515: Paper No. 126387.



\bibitem{SS}
E. M. Stein, R. Shakarchi, \emph{Complex Analysis}, Princeton Lectures in Analysis, vol. 2, Princeton University Press, Princeton, NJ, 2003.

\bibitem{UK}
S. I. Ueki, Weighted composition operator on the Fock space, \emph{Proc. Amer. Math. Soc.}, 2007, 135: 1405--1410.

\bibitem{Z}
L. Zhao, Unitary weighted composition operators on the Fock space of $\C^n$, \emph{Complex Anal. Oper. Theory}, 2014, 8: 581--590.

\bibitem{ZK}
K. Zhu, \emph{Analysis on Fock Spaces}, Graduate Texts in Mathematics, vol. 263, Springer, New York, 2012.

\end{thebibliography}
\end{document}